\documentclass[preprint]{elsarticle}

\usepackage{amsmath}
\usepackage{amsfonts}
\usepackage{amssymb}
\usepackage{amsthm}
\usepackage[utf8]{inputenc}  
\usepackage[T1]{fontenc}     
\usepackage{hyperref}
\usepackage{yhmath}
\usepackage{algorithm,algpseudocode}
\usepackage{multirow}
\usepackage{mathtools}
\usepackage{graphicx}
\usepackage{mathtools}
\usepackage{tikz}
\usepackage{stmaryrd}
\usepackage{amsopn}
\usepackage{stmaryrd}
\usepackage{caption}
\usepackage{float}
\usepackage{empheq}

\newtheorem{remark}{\it Remark}[section]
\newtheorem{example}{\it Example}[section]
\newtheorem{definition}{\it Definition}[section]
\newtheorem{proposition}{\it Proposition}[section]
\newtheorem{theorem}{\it Theorem}[section]
\newtheorem{lemma}{\it Lemma}[section]
\newtheorem{corollary}{\it Corollary}[section]

\makeatletter
\DeclareFontFamily{U}{tipa}{}
\DeclareFontShape{U}{tipa}{m}{n}{<->tipa10}{}
\newcommand{\arc@char}{{\usefont{U}{tipa}{m}{n}\symbol{62}}}%
\newcommand{\arc}[1]{\mathpalette\arc@arc{#1}}
\newcommand{\arc@arc}[2]{%
	\sbox0{$\m@th#1#2$}%
	\vbox{
		\hbox{\resizebox{\wd0}{\height}{\arc@char}}
		\nointerlineskip
		\box0
	}%
}
\makeatother
\newcommand*\curvearc[2][1em]{\overset{\tikz[trim left]\draw[->](0,0)to[bend left]({#1},0);}{#2}}

\algnewcommand{\Inputs}[1]{%
	\State \textbf{Inputs:}
	\Statex \hspace*{\algorithmicindent}\parbox[t]{.8\linewidth}{\raggedright #1}
}
\algnewcommand{\Initialize}[1]{%
	\State \textbf{Initialize:}
	\Statex \hspace*{\algorithmicindent}\parbox[t]{.8\linewidth}{\raggedright #1}
}

\makeatletter
\def\ps@pprintTitle{%
	\let\@oddhead\@empty
	\let\@evenhead\@empty
	\def\@oddfoot{}%
	\let\@evenfoot\@oddfoot}
\makeatother

\begin{document}
	
	\title{\textit{Simuorb}: a new method for generating and describing the intersection points of clique-arrangements}
	
	\author[LMPA]{P.~Ryckelynck\corref{cor1}}
	\ead{ryckelyn@lmpa.univ-littoral.fr}
	\author[LMPA]{L.~Smoch}
	\ead{smoch@lmpa.univ-littoral.fr}
	\address[LMPA]{ULCO, LMPA, F-62100 Calais, France\\Univ Lille Nord de France, F-59000 Lille, France. CNRS, FR 2956, France.}
	
	\begin{abstract}
		This work, which may be seen as a companion paper to \cite{RS2}, handles the way the intersection points made by the diagonals of a regular polygon are distributed. It was stated recently by the authors that these points lie exclusively on circles centered on the origin and also the way their respective radii depend on the four indices of the vertices of the initial regular $n$-gon which characterize the two straight lines underlying the intersection points. Because these four vertices are located at preset positions on the  the regular $n$-gon inscribed in the unit circle whose path-length perimeter is constant, it allows the orbits to be characterized by 3 parameters instead of 4, describing roughly the lengths of the paths between the first three vertices, whether the quadrilateral described by these four vertices is simple or complex. 
		This approach enables us to deal with the orbits generated by the clique-arrangement, and to handle their cardinalities as well as the multiplicities of the associated intersection points. A reliable counting-algorithm based on this triplet strategy is provided in order to enumerate the intersection points without generating the associated graph. The orbits being simulated, we call this method \textit{Simuorb}. The procedure is robust, fast and allows a comprehensive understanding of what is happening in a clique-arrangement, whether it contains a large number of points or not.
	\end{abstract}
	
	
	\maketitle
	
\section{Introduction, notation and historical notice}

The interest aroused by the question of description, distribution and enumeration of intersection points as well as polygons generated from an arrangement of lines (and in particular from the diagonals of a regular $n$-gon) is obvious and many researchers including G. Bol (\cite{Bol}), H. Steinhaus (\cite{St1, St2}), H. Harborth (\cite{Ha1, Ha2}), C.E. Tripp (\cite{Tr}), J.F. Rigby (\cite{Rig}) have focused on this problem.\\
An inspiring paper is that of Poonen and Rubinstein \cite{PR}, in which the authors are the first to provide two remarkable formulas for the number $N_{i,n}$ of \textit{interior intersection points} made by the diagonals of a regular $n$-gon and the associated number of regions.\\
The final outputs of the formulas provided by Poonen and Rubinstein are polynomials on each residue class modulo 2520. Here are the first values of $N_{i,n}$ referenced as A006561 in OEIS \cite{OEIS}.
\begin{center}
	\begin{tabular}{|c|*{13}{c|}}
		\hline
		$n$ & 3 & 4 & 5 & 6 & 7 & 8 & 9 & 10 & 11 & 12 & 13 & 14 & \ldots\\
		\hline
		$N_{i,n}$ & 0 & 1 & 5 & 13 & 35 & 49 & 126 & 161 & 330 & 301 & 715 & 757 & \ldots\\
		\hline
	\end{tabular}
\end{center}
The authors compute naturally the number of regions formed by the diagonals, by using Euler's formula $V-E+F=2$.\\
Nevertheless, the formulas provided by Poonen and Rubinstein in \cite{PR} as remarkable as they are, deal only with the interior intersection points and do not focus on the geometrical way the internal and the external points are distributed. Though, by subtracting the previous values to the total points numbers $N_n$ of the graph (referenced in OEIS as A146212) and the $n$ points on the unit circle, i.e. by calculating $N_{e,n}=N_n-n-N_{i,n}$, we get the sequence of exterior points numbers of the cliques arrangements referenced also in OEIS as A146213 
\begin{center}
	\begin{tabular}{|c|*{13}{c|}}
		\hline
		$n$ & 3 & 4 & 5 & 6 & 7 & 8 & 9 & 10 & 11 & 12 & 13 & 14 & \ldots\\
		\hline
		$N_{e,n}$ & 0 & 0 & 5 & 18 & 49 & 88 & 198 & 300 & 550 & 588 & 1235 & 1414 & \ldots\\
		\hline
	\end{tabular}
\end{center}
This sequence is defined, for $n$ odd given only, through the formula $N_{e,n}=n(2n^3-15n^2+34n-21)/24$. So far, the formula for $n$ even is not known because the starting point of the Poonen and Rubinstein's formula, i.e. the concurrency of diagonals inside the unit circle and the use of similar triangles, does not apply outside the unit circle.\\ 
Concerning the geometrical location of the intersection points, whether interior or exterior, we stated in a recent paper \cite{RS2} that all the intersection points are located on \textit{circular orbits} centered at the origin, and that the radii of these orbits are defined by the function $J_n$ defined below. We propose in this paper to study more precisely the way the points are distributed on these orbits and to explain why their cardinalities and their multiplicities may be different. But before presenting this radius/modulus function, let us provide some notation.\\

Let $n$ be an integer and let $\mathcal{C}_n$ be the regular $n$-gon with vertices $z_m=\exp \left(m\frac{2\pi}{n}I\right)$, where $0\leq m<n$ and $I=\sqrt{-1}$. We use throughout the paper the following definitions and assumptions. For any set $S$, $\sharp(S)$ denotes its cardinality. If $f$ is any mapping,  $\Im(f)$ denotes its range. As usual, if $x\in\mathbb{R}$ then $[x]$ denotes the greatest integer function. Let $i,j,k,\ell$ be four indices in $\llbracket 0,n-1\rrbracket$ with $\sharp\{i,j\}=\sharp\{k,\ell\}=2$. When the two straight-lines $\mathcal{D}_{i,j}$ and $\mathcal{D}_{k,\ell}$ defined respectively by the two couples $(z_i,z_j)$ and $(z_k,z_\ell )$ are secant, we denote $z_{i,j,k,\ell}$ as their intersection point. Let $\mathcal{K}_n$ the \textit{geometric graph} which consists of
the union of the vertices of $\mathcal{C}_n$ and the intersection points $z_{i,j,k,\ell }$, say $V(\mathcal{K}_n)$, together with the sets of straight lines $\mathcal{D}_{i,j}$, say $E(\mathcal{K}_n)$. We partition $\mathcal{K}_n$ in three sets as follows $\mathcal{K}_n=\mathcal{C}_n\cup \mathcal{K}_{e,n}\cup \mathcal{K}_{i,n}$, by defining respectively $z\in \mathcal{K}_{e,n}$ and $z\in\mathcal{K}_{i,n}$ if and only if 
$\left| z\right| >1$ and $\left| z\right| <1$. The arrangement of lines in the plane which consists of the various straight lines $\mathcal{D}_{i,j}$, passing through the points $z_i$ and $z_j $ for all indices $i\neq j$, may be referred to as the \textit{clique-arrangement} constructed from $\mathcal{C}_n$.\\
Let us use the abbreviation $pq_{\pm}=(p\pm q)\frac \pi n$ for convenient integers $p$ and $q$. Although we may use indices modulo $n$, we impose instead the inequalities $0\leq i\leq n-1$ over all subsequent indices.
\\
We will denote by $M_{e,n}$ and $M_{i,n}$ the respective numbers of exterior and interior orbits of the arrangement $\mathcal{K}_n$. We set $M_n=1+M_{e,n}+M_{i,n}$ which stands for the total number of orbits. Likewise, we will denote by $N_{e,n}$ and $N_{i,n}$ the respective cardinalities of the two sets $\mathcal{K}_{e,n}$ and $\mathcal{K}_{i,n}$. Let $N_n=n+N_{e,n}+N_{i,n}$ the whole number of intersection points.\\
In the following table we give for $3\leq n\leq 10$, the values of cardinalities $N_n,N_{e,n},N_{i,n}$, and the numbers of various orbits $M_n,M_{e,n},M_{i,n}$.
\begin{table}[!h]
	{\small\centering
		\begin{tabular}{|c||c|c|c||c|c|c|}
			\hline
			$n$ & $N_n$ & $N_{e,n}$ & $N_{i,n}$ & $M_n$ & $M_{e,n}$ & $M_{i,n}$ \\ \hline
			3 & 3 & 0 & 0 & 1 & 0 & 0 \\ \hline
			4 & 5 & 0 & 1 & 2 & 0 & 1 \\ \hline
			5 & 15 & 5 & 5 & 3 & 1 & 1 \\ \hline
			6 & 37 & 18 & 13 & 6 & 2 & 3 \\ \hline
			7 & 91 & 49 & 35 & 10 & 5 & 4 \\ \hline
			8 & 145 & 88 & 49 & 14 & 7 & 6 \\ \hline
			9 & 333 & 198 & 126 & 25 & 14 & 10 \\ \hline
			10 & 471 & 300 & 161 & 32 & 18 & 13 \\ \hline
		\end{tabular}
		\caption{Numbers of intersection points and orbits for $3\leq n\leq 10$}}
\end{table}

As shown in Figure \ref{fig1}, it becomes very difficult when $n$ becomes large to understand the way the orbits behave, how are they separated from each other and how the intersection points are distributed. It will also be noted that the number of points on each orbit may be different.
\begin{figure}[!ht]
	\begin{center}
		\includegraphics[width=10cm,height=7cm]{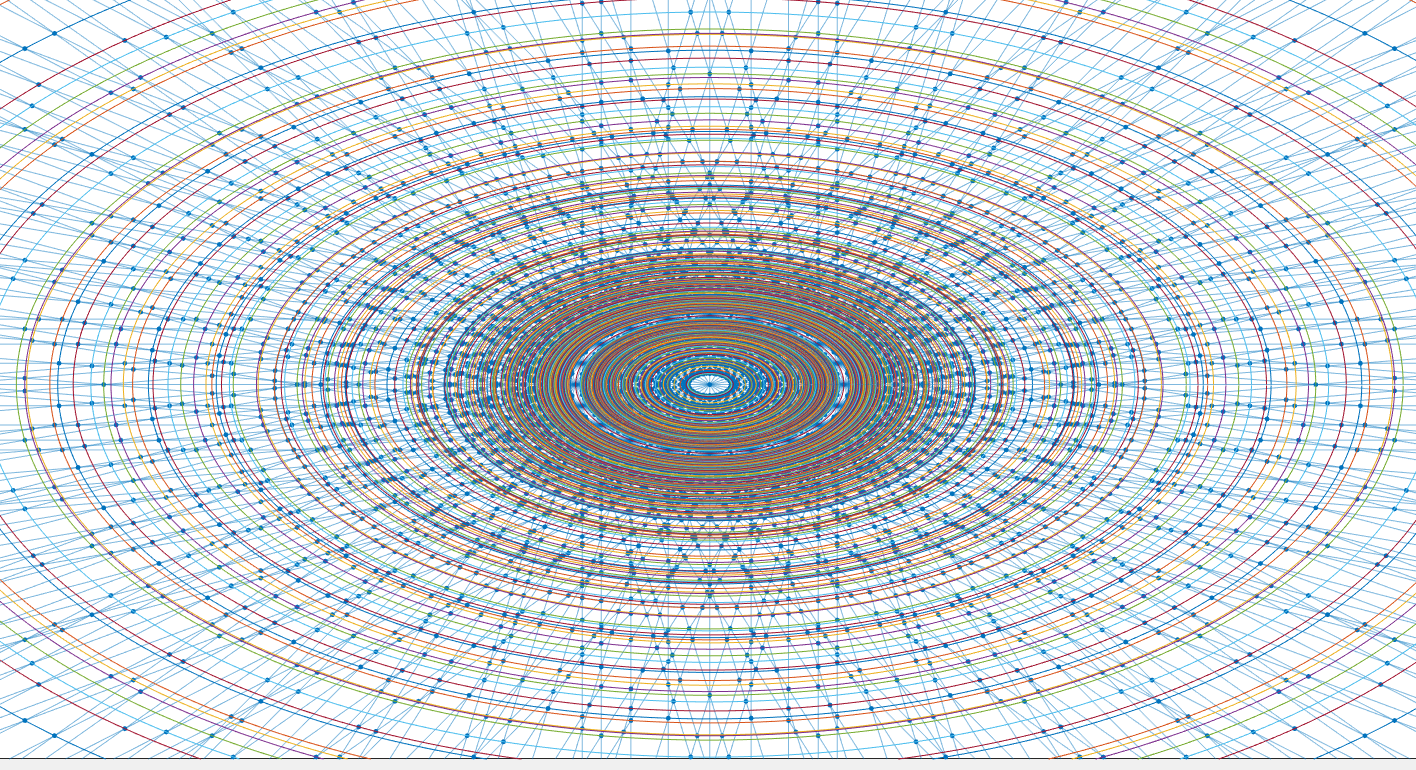}
	\end{center}
	\caption{\centering A partial view of the cyclotomic clique-arrangement for $n=20$ together with its circular orbits}
	\label{fig1}
\end{figure}

Therefore, the aim of this work is to lighten the phenomena underlying the intersection of straight lines generated from the vertices of a regular $n$-gon.\\

The rest of this paper is organized as follows. We introduce in Section 2 the fundamental parameters $p,q,r$ and $s$ which characterize the shape of an inscribed quadrilateral whether it is simple or complex as well as some important properties about these integers. Section 3 is devoted to the square radius $J_n$ of the circular orbits containing all the intersection points of the clique-arrangement and its description with respect to the three paramters $p,q,r$. The fourth section is related to the arc length between two points lying on the same orbit in order to state the exact cardinality of each orbit and to determine the multiplicity of their points. The notion of equivalence between the triplets $(p,q,r)$ is defined thus, allowing to work with the orbits without redundancy. Section 5 deals with the different procedures which make up the body of the algorithm called \textit{Simuorb}. The sixth and last section presents some numerical considerations relative to the efficiency of \textit{Simuorb}.

\section{Triplets to characterize cyclic quadrilaterals}

The description of a vertex $z_{i,j,k,\ell}\in\mathcal{K}_n$ (that we assume of multiplicity 2 for instance) help to $\mathcal{D}_{i,j}$ and $\mathcal{D}_{k,\ell}$ or help to the pair of pairs $\{\{i,j\},\{k,\ell\}\}$ is not unambiguously defined unlike that provided by the quadrilateral $\mathcal{Q}$ formed from the two edges $(z_i,z_j)$ and $(z_k,z_\ell)$, and which may write under 16 different ways by enforcing the vertices $z_i$ and $z_j$ to hold two successive positions in the quadruplet. From now on, throughout the rest of this paper, \textit{we constraint the first two components of any quadrilateral to be associated to a first segment among $[z_i,z_j]$ and $[z_k,z_\ell]$ and the last two components to the remaining one}, which halves the previous number of possibilities. Any inscribed quadrilateral generates in general (if the straight lines are not parallel) two intersection points located outside the unit circle if $\mathcal{Q}$ is simple and on either side of it otherwise. The nature of $\mathcal{Q}$ depends obviously on the way the points $z_i,z_j,z_k$ and $z_\ell$ are distributed on the unit circle but in each case, the 4 quadrilaterals $(z_i,z_j,z_k,z_\ell)$, $(z_k,z_\ell,z_i,z_j)$, $(z_j,z_i,z_\ell,z_k)$ and $(z_\ell,z_k,z_j,z_i)$ are equal. 
\begin{figure}[!ht]
  \tabcolsep=0.2\linewidth
\divide\tabcolsep by 8
\begin{tabular*}{\textwidth}{@{\extracolsep{\fill}}ccc}
	\includegraphics[width=0.30\textwidth]{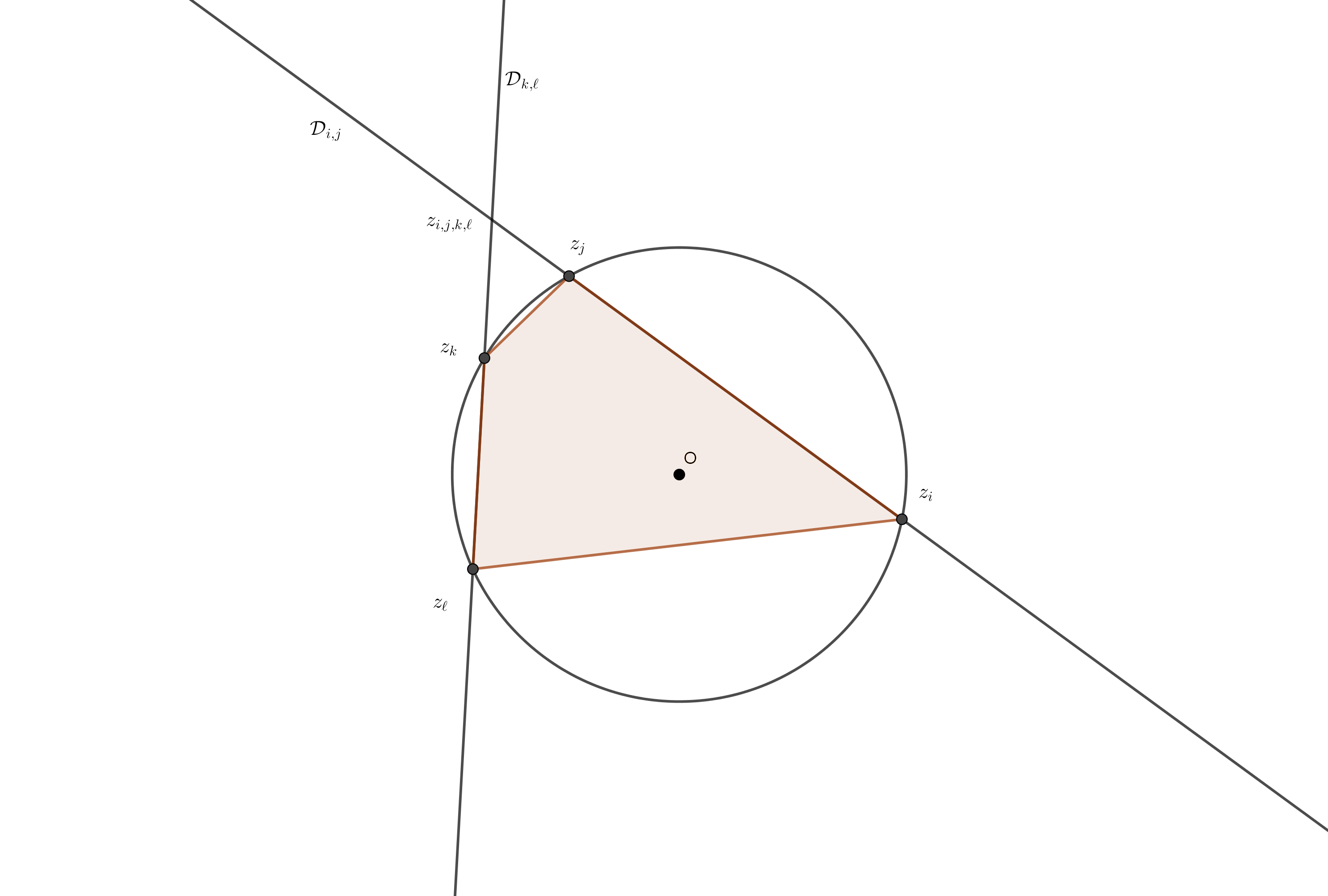} &
	\includegraphics[width=0.30\linewidth]{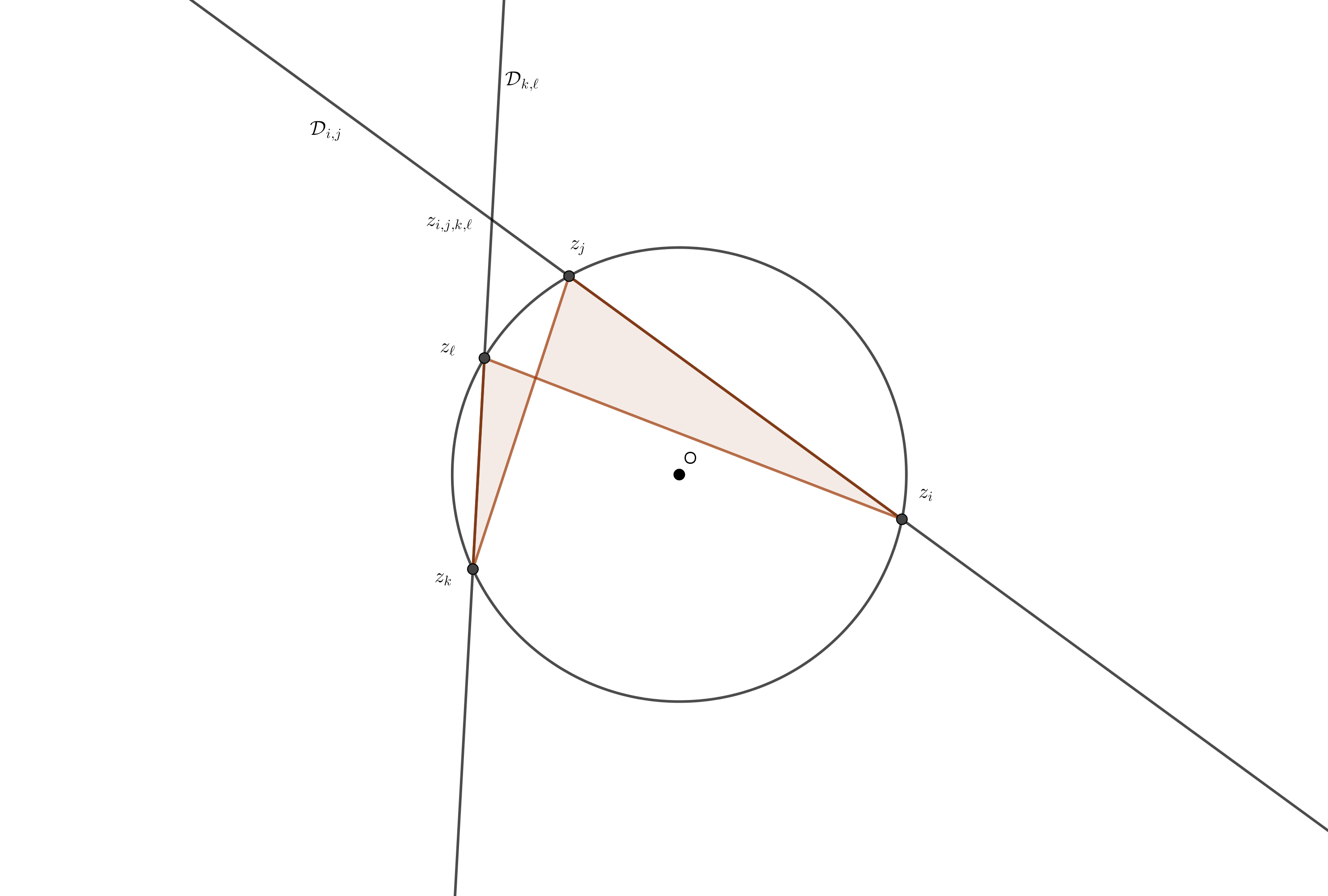} &
	\includegraphics[width=0.30\linewidth]{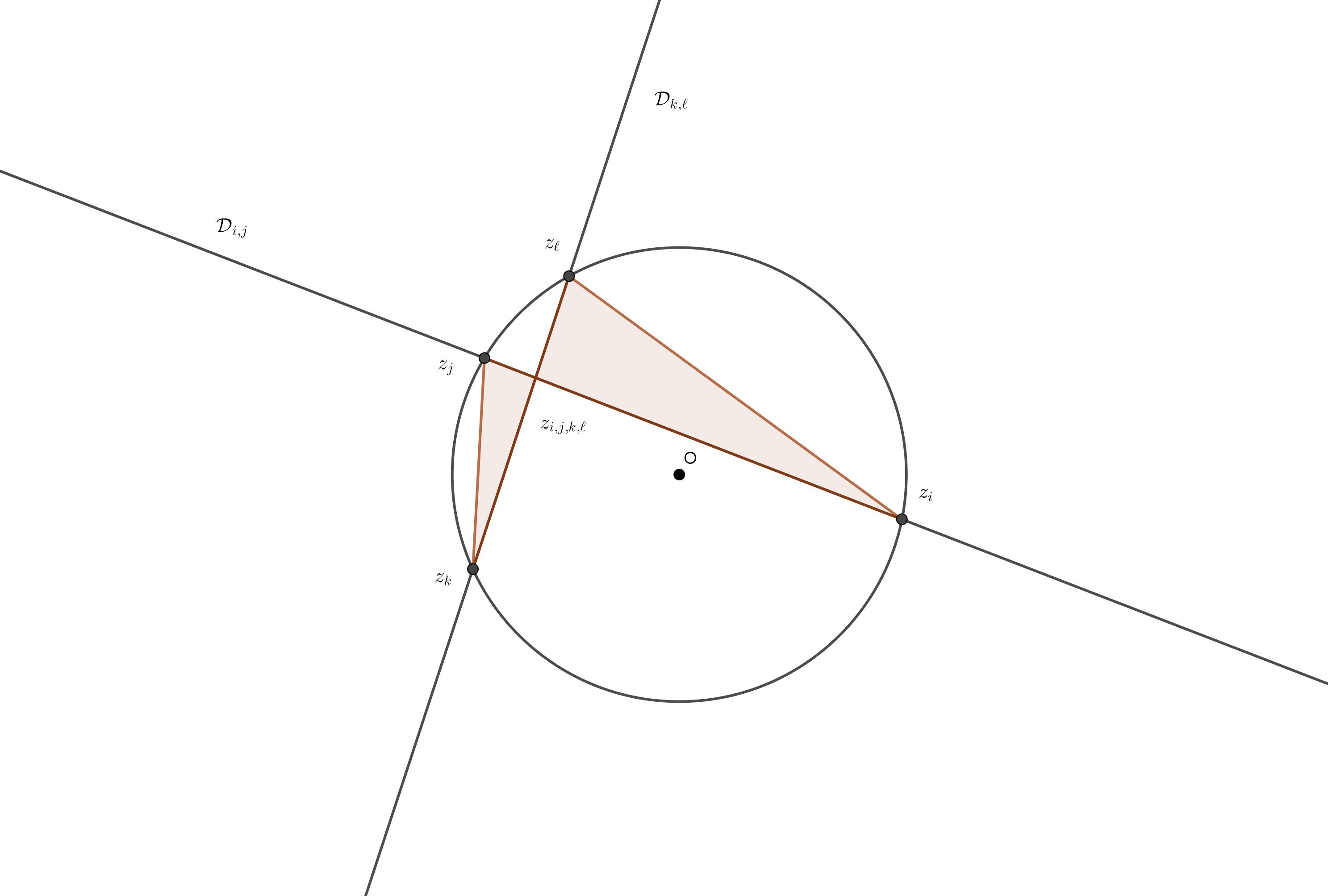}
	\\
	\footnotesize Case 1 & \footnotesize Case 2 & \footnotesize Case 3
\end{tabular*}
\caption{The three configurations generating $z_{i,j,k,\ell}=\mathcal{D}_{i,j}\cap\mathcal{D}_{k,\ell}$ either outside or inside the unit circle}
\label{figquad}
\end{figure}

We present in Figure \ref{figquad} the three ways to generate  $z_{i,j,k,\ell}=\mathcal{D}_{i,j}\cap\mathcal{D}_{k,\ell}$ whether it is outside or inside the unit circle, by considering carefully the condition mentioned previously relative to the location of $z_i$ and $z_j$ in the quadruplet. There is no other possible configuration. Therefore, if we generate all the possible inscribed quadrilaterals $(z_i,z_j,z_k,z_\ell)$, $i,j,k,\ell\in\{0,\ldots,n-1\}$ according to the 3 configurations described in Figure \ref{figquad}, we generate all the intersection points of the clique-arrangement even if there shall be redundancy.\\

Now, in order to get an effective correspondence between the intersection points and the different quadrilaterals, we need to characterize a quadrilateral in another way than just using an ordered constrained sequence of vertices. Because the quadrilateral is inscribed in the unit circle and its vertices are located at preset positions on the regular cyclic $n$-gon, we consider the lengths of the paths between the vertices rather than the vertices themselves or what amounts to the same thing, we consider a shape of quadrilaterals rather than the quadrilateral itself. Because the sum of the 4 path-lengths is equal to $n$, a quadrilateral may be characterized by a triplet of path-lengths whose constrained amplitude will allow to provide help to a triple loop algorithm a process for generating all the orbits of intersection points. \\  

We equip the regular $n$-gon $\mathcal{C}_n$ with the structure of a directed geometric graph, say $\widetilde{\mathcal{C}_n}$, by considering the set of vertices $\{z_m=e^{m\frac{2\pi}{n}I}\}_{m\in\{0,\ldots,n-1\}}$ all located on the unit circle, and the set of directed edges $\{(z_m,z_{m+1})\}_{m\in\{0,\ldots,n-1\}}$ connecting every vertex $z_m$ to its neighbour $z_{m+1}$, indices taken modulo $n$. We shall denote by the way $\curvearc{z_iz_j}$ the polygonal directed path from the vertex $z_i$ to the vertex $z_j$. Next, we define as $\delta$ the mapping which associates to the pair $(i,j)$, $i\neq j$, the length of the path $\curvearc{z_iz_j}$ according to the graph $\widetilde{\mathcal{C}_n}$, i.e.
\begin{equation}
	\delta(i,j)=\left\{\begin{array}{rcl}j-i & \mbox{ if }i<j\\ n+j-i & \mbox{ otherwise}\end{array}\right..
	\label{delta}
\end{equation} 
Clearly, $1\leq\delta(i,j)\leq n-1$. We remark also that $\delta(i,j)+\delta(j,i)=n$ and then, $\delta(i,j)$ and $-\delta(j,i)$ are congruent modulo $n$ without having in general the same absolute value. Consequently, we may define as $\Delta$ the mapping
\begin{equation}
(i,j)\rightarrow\Delta(i,j)=\underset{\alpha\in\{\delta(i,j),-\delta(j,i)\}}{\mbox{argmin}}\{|\alpha|\}.
\label{deltaretrtilde}
\end{equation}
Unlike the mapping $\delta$ which provides only positive path lengths, the mapping $\Delta$ provides the smallest path length whether positive or negative, according to the graph $\widetilde{\mathcal{C}_n}$, between two vertices $z_i$ and $z_j$.\\
\\
Let us recall from \cite{RS2} that we may suppose that the straight lines $\mathcal{D}_{i,j}$ and $\mathcal{D}_{k,\ell}$ are not vertical nor parallel and thus, that none of the three integers $i+j$, $k+\ell$ and $i+j-(k+\ell)$ are equal to 0 modulo $n$. Moreover, $i\neq j$, $k\neq\ell$ and $\sharp(\{i,j,k,\ell\})\geq 3$.\\
We have the following result, the proof of which is obvious.
\begin{lemma}
	The vertex $z_{i,j,k,\ell}\in V(\mathcal{K}_n)$ does not lie on $\mathcal{C}_n$ if and only if $\sharp(\{i,j,k,\ell\})=4$.
	\label{lemCard3}
\end{lemma}
From now on, we suppose that $\sharp(\{i,j,k,\ell\})=4$. Given a pair of pairs $\mathcal{P}=\{\{i,j\},\{k,\ell\}\}$, we may consider the directed cycles along the regular $n$-gon $\mathcal{C}_n$ preserving the connectivity encoded by $\mathcal{P}$. By using the definitions of simple and complex quadrilaterals, we have the following lemma which is straightforward.
\begin{proposition}
	Let us consider $\mathcal{P}$ as well as the quadrilateral $\mathcal{Q}$ associated to the path $\Gamma=(z_i,z_j,z_k,z_\ell,z_i)$ in $\widetilde{\mathcal{C}}_n$. Then, $\mathcal{Q}$ is simple if and only 
	\begin{center}
	$\sharp\left(\curvearc{z_i z_j}\cap\{z_k,z_\ell\}\right)=\sharp\left(\curvearc{z_k z_\ell}\cap\{z_i,z_j\}\right)=0\mbox{ or }2$
	\end{center}
	while $\mathcal{Q}$ is complex if and only if 
	\begin{center}
		$\sharp\left(\curvearc{z_i z_j}\cap\{z_k,z_\ell\}\right)=\sharp\left(\curvearc{z_k z_\ell}\cap\{z_i,z_j\}\right)=1$ or $\sharp\left(\curvearc{z_i z_j}\cap\{z_k,z_\ell\}\right)\neq \sharp\left(\curvearc{z_k z_\ell}\cap\{z_i,z_j\}\right)$.
	\end{center} 
	Moreover,
	\begin{itemize}
		\item If $\sharp\left(\curvearc{z_i z_j}\cap\{z_k,z_\ell\}\right)=\sharp\left(\curvearc{z_k z_\ell}\cap\{z_i,z_j\}\right)=1$, the point $z_{i,j,k,\ell}$ is inside the unit circle.
		\item If $\sharp\left(\curvearc{z_i z_j}\cap\{z_k,z_\ell\}\right)\neq \sharp\left(\curvearc{z_k z_\ell}\cap\{z_i,z_j\}\right)$, the point $z_{i,j,k,\ell}$ is outside the unit circle.
	\end{itemize}
	\label{prop2.1}
\end{proposition}
\begin{proof}
	When $\mathcal{Q}=(z_i,z_j,z_k,z_\ell)$ is simple, the straight lines $\mathcal{D}_{i,j}$ and $\mathcal{D}_{k,\ell}$ cross outside the unit circle  (as well as the straight lines $\mathcal{D}_{j,k}$ and $\mathcal{D}_{\ell,i}$). Then it excludes the case $\sharp\left(\curvearc{z_i z_j}\cap\{z_k,z_\ell\}\right)=\sharp\left(\curvearc{z_k z_\ell}\cap\{z_i,z_j\}\right)=1$. Let us assume that  $\sharp\left(\curvearc{z_i z_j}\cap\{z_k,z_\ell\}\right)\neq \sharp\left(\curvearc{z_k z_\ell}\cap\{z_i,z_j\}\right)$. If $\sharp\left(\curvearc{z_i z_j}\cap\{z_k,z_\ell\}\right)=0$ then $\sharp\left(\curvearc{z_k z_\ell}\cap\{z_i,z_j\}\right)=2$ which means that $\mathcal{Q}$ is complex which is impossible, the same conclusion holding if $\sharp\left(\curvearc{z_i z_j}\cap\{z_k,z_\ell\}\right)=2$. 
	By contraposition we get the second result. In this last case, if $\sharp\left(\curvearc{z_i z_j}\cap\{z_k,z_\ell\}\right)=\sharp\left(\curvearc{z_k z_\ell}\cap\{z_i,z_j\}\right)=1$ then $z_k$ and $z_\ell$ are on both sides of the straight line formed by $z_i$ and $z_j$. Then $\mathcal{D}_{i,j}$ and $\mathcal{D}_{k,\ell}$ cross inside the unit circle. If $\sharp\left(\curvearc{z_i z_j}\cap\{z_k,z_\ell\}\right)\neq \sharp\left(\curvearc{z_k z_\ell}\cap\{z_i,z_j\}\right)$, $\mathcal{D}_{i,j}$ and $\mathcal{D}_{k,\ell}$ cross inside the unit circle which ends the proof.
\end{proof}

\begin{definition}
	When the quadrilateral is simple, we may rearrange (if necessary) the path $\Gamma=(z_i,z_j,z_k,z_\ell,z_i)$ as a new path $\Gamma=(z_{i'},z_{j'},z_{k'},z_{\ell'},z_{i'})$ such that
	\begin{center}
		$\{i,j\}=\{i',j'\},~~\{k,\ell\}=\{k',\ell'\},~~\sharp\left(\curvearc{z_{i'}z_{j'}}\cap\{z_{k'},z_{\ell'}\}\right)=\sharp\left(\curvearc{z_{k'}z_{\ell'}}\cap\{z_{i'},z_{j'}\}\right)=0$.
	\end{center}
	Such a quadruplet $(i',j',k',\ell')$ is called admissible. 
	\label{defadm}
\end{definition}
\noindent This being done, the four directed paths $\curvearc{z_{i'}z_{j'}}$, $\curvearc{z_{j'}z_{k'}}$, $\curvearc{z_{k'}z_{\ell'}}$ and $\curvearc{z_{\ell'}z_{i'}}$ have trivial intersections at the vertices. No such a criterium for quadruplets associated to complex quadrangles does exist and we will say in that case that such any quadruplet is naturally admissible. Then, when the quadrilateral is complex, we define somewhat arbitrarily $(i',j',k',\ell')=(i,j,k,\ell)$.
\begin{remark}
	For any admissible quadruplet, with notation of Figure \ref{figquad}, we have the following correspondences:
	\begin{itemize}
		\item Case 1: $\sharp\left(\curvearc{z_i z_j}\cap\{z_k,z_\ell\}\right)=\sharp\left(\curvearc{z_k z_\ell}\cap\{z_i,z_j\}\right)=0$,
		\item Case 2: $\sharp\left(\curvearc{z_i z_j}\cap\{z_k,z_\ell\}\right)\neq \sharp\left(\curvearc{z_k z_\ell}\cap\{z_i,z_j\}\right)$,
		\item Case 3: $\sharp\left(\curvearc{z_i z_j}\cap\{z_k,z_\ell\}\right)=\sharp\left(\curvearc{z_k z_\ell}\cap\{z_i,z_j\}\right)=1$.
	\end{itemize}
	\label{remcase}
\end{remark}
\begin{lemma}
Let $(i,j,k,\ell)$ be an admissible quadruplet and $\mathcal{Q}$ its associated quadrilateral. Then $\mathcal{Q}$ is simple if and only if
	\begin{equation}
		\delta(i,j)+\delta(j,k)+\delta(k,\ell)+\delta(\ell,i)=n 
		\label{longueursimple}
	\end{equation}
	while $\mathcal{Q}$ is a complex quadrilateral if and only if
	\begin{equation}
		\delta(i,j)+\delta(j,k)+\delta(k,\ell)+\delta(\ell,i)=2n.
		\label{longueurcomplexe}
	\end{equation}
\label{lemma2.2}
\end{lemma}
\begin{proof}
	In the case of an admissible quadruplet $(i,j,k,\ell)$ and its associated simple quadrilateral, the path $\Gamma=(z_i,z_j,z_k,z_\ell,z_i)$ is a  Hamiltonian cycle since it consists of a sequence of adjacent and distinct vertices in $\mathcal{C}_n$, which starts and ends at the same vertex $z_{i}$ without repeated edges. Therefore, according to $\widetilde{\mathcal{C}_n}$, its length is equal to $n$ and the equality (\ref{longueursimple}) holds. When $\mathcal{Q}$ is a complex quadrilateral, $\Gamma$ is not anymore a  Hamiltonian cycle. Two cases may occur. First, if $\sharp\left(\curvearc{z_i z_j}\cap\{z_k,z_\ell\}\right)=0$ and $\sharp\left(\curvearc{z_k z_\ell}\cap\{z_i,z_j\}\right)=2$, then 
	\begin{center}
		$\delta(i,j)+\delta(j,k)+\delta(k,\ell)+\delta(\ell,i)=\delta(i,j)+(\delta(j,\ell)+\delta(\ell,k))+(\delta(k,i)+\delta(i,j)+\delta(j,\ell))+(\delta(\ell,k)+\delta(k,i))$\\
		$=2(\delta(i,j)+\delta(j,\ell)+\delta(\ell,k)+\delta(k,i))=2n$.
	\end{center}
	The same result holds when $\sharp\left(\curvearc{z_i z_j}\cap\{z_k,z_\ell\}\right)=2$ and $\sharp\left(\curvearc{z_k z_\ell}\cap\{z_i,z_j\}\right)=0$.
	Second, the vertices $z_{k}$ and $z_{\ell}$ are on both sides of the straight line formed by $z_{i}$ and $z_{j}$. If $z_{k}\in\curvearc{z_{i}z_{j}}$ and $z_{\ell}\in\curvearc{z_{j}z_{i}}$ we have $\delta(i,j)=\delta(i,k)+\delta(k,j)=\delta(i,k)+n-\delta(j,k)$ and $\delta(k,\ell)=n-\delta(\ell,k)=n-(\delta(\ell,i)+\delta(i,k))$. Thus, 
	\begin{center}
		$\delta(i,j)+\delta(j,k)+\delta(k,\ell)+\delta(\ell,i)=(\delta(i,k)+n-\delta(j,k))+\delta(j,k)+(n-(\delta(\ell,i)+\delta(i,k)))+\delta(\ell,i)=2n$.
	\end{center}
	If $z_{k}\in\curvearc{z_{j}z_{i}}$ and $z_{\ell}\in\curvearc{z_{i}z_{j}}$, we proceed exactly in the same way which ends the proof.
\end{proof}
As a consequence of this lemma, we may state the following result.
\begin{lemma}
	Let $(i,j,k,\ell)$ be an admissible quadruplet and $\mathcal{Q}=(z_{i},z_{j},z_{k},z_{\ell})$ its associated quadrilateral. We set 
	\begin{center}
		$\mathcal{S}=\{(a,b)\in\{(i,j),(j,k),(k,\ell),(\ell,i)\},~\delta(a,b)=n+b-a\}$.
	\end{center}
	 Then $\mathcal{Q}$ is simple (resp. complex) if and only if $\sharp(\mathcal{S})=1$ (resp. $\sharp(\mathcal{S})=2$).
	\label{lemma2.3}
\end{lemma}
\begin{proof}
	The proof is straightforward. Let us mention first that $\delta(a,b)=n+b-a$ if and only if $z_a$ and $z_b$ are on either side of $z_0$. Then, by using the results of Lemma \ref{lemma2.2}, we pass through the point $z_0$ once when $\mathcal{Q}$ is simple, while $z_0$ is crossed twice when $\mathcal{Q}$ is complex.
\end{proof}
\begin{example} Let us consider the case $n=12$ and the vertices $z_3$, $z_4$, $z_5$ and $z_7$. For example, the vertex $z_{3,4,5,7}$ located outside the unit circle may be characterized through the quadruplets $(3,4,5,7)$, $(4,3,7,5)$, $(5,7,3,4)$, $(7,5,4,3)$, $(3,4,7,5)$, $(4,3,5,7)$, $(5,7,4,3)$ and $(7,5,3,4)$, the first four being associated to simple quadrilaterals and the last four to complex ones. Among those associated to simple quadrilaterals, the quadruplets $(3,4,5,7)$ and $(5,7,3,4)$ are the only ones admissible. Indeed, for these two quadruplets, $\sharp(\mathcal{S})=1$, $\sharp\left(\curvearc{z_{3}z_{4}}\cap\{z_5,z_7\}\right)=\sharp\left(\curvearc{z_{5}z_{7}}\cap\{z_3,z_4\}\right)= 0$ and  $\sharp\left(\curvearc{z_{5}z_{7}}\cap\{z_3,z_4\}\right)=\sharp\left(\curvearc{z_{3}z_{4}}\cap\{z_5,z_7\}\right)= 0$. The length of the paths associated to these two quadruplets is equal to $n=12$. On the other hand, the vertex $z_{4,7,3,5}$ located inside the unit circle is characterized by the quadruplets $(4,7,3,5)$, $(4,7,5,3)$, $(7,4,3,5)$, $(7,4,5,3)$, $(3,5,4,7)$, $(3,5,7,4)$, $(5,3,4,7)$ and $(5,3,7,4)$ which are all necessarily associated to complex quadrilaterals and are all admissible. Regarding the quadrangle $(z_4,z_7,z_3,z_5)$ for example, we note that $\delta(4,7)+\delta(7,3)+\delta(3,5)+\delta(5,4)=3+8+2+11=24$ and $\sharp(\mathcal{S})=2$.
\end{example}
Now that we have detailed some characteristics of quadrilaterals, we may propose a way to characterize them other than by their vertices.
\begin{definition}
If $(i,j,k,\ell)$ denotes an admissible quadruplet and $\mathcal{Q}=(z_{i},z_{j},z_{k},z_{\ell})$ is its associated quadrilateral, we set
\begin{equation}
	p=\delta(i,j),\qquad q=\delta(k,\ell),\qquad
	r=\left\{\begin{array}{rcl}
		\delta(j,k) & \mbox{if} & \sharp(\mathcal{S})=1\\
		\Delta(j,k) & \mbox{if} & \sharp(\mathcal{S})=2
		\end{array}\right.,\qquad s=n-(p+q+r).
	\label{defpqrs}		
\end{equation}
\label{ddefpqrs}
\end{definition}
While $p$ and $q$ inform us about the path-lengths of the segments $[z_i,z_j]$ and $[z_k,z_\ell]$, the parameter $r$ provides the relative smallest path-length between $z_j$ and $z_k$ and therefore an indication on the gap between the two previous segments. As a consequence of these definitions we have the following result.
\begin{lemma}
Let be given 4 integers $i,p,q,r$ with $1\leq p,q\leq n-1$. The solution of equation (\ref{defpqrs}) is as follows
\begin{center}
	$(i,j,k,\ell)=(i,i+p,i+p+r,i+p+q+r)$.
\end{center}
\label{lemma2.4}
\end{lemma} 
Clearly, to a triplet $(p,q,r)$ correspond $n$ quadrilaterals which are obtained from each other by rotation of center the origin and of angle $\frac{2\pi}{n}$. For each of them, only the intersection point of $\mathcal{D}_{i,i+p}$ and $\mathcal{D}_{i+p+r,i+p+q+r}$ is highlighted since $p$ and $q$ are necessarily associated to these two specific straight lines. If we are concerned with the second intersection point $\mathcal{D}_{i+p,i+p+r}\cap\mathcal{D}_{i+p+q+r,i}$, we will have to consider another triplet $(\tilde{p},\tilde{q},\tilde{r})$, $\tilde{p}$ and $\tilde{q}$ being necessarily connected to these two straight lines.\\
We note also that the condition of Lemma \ref{lemCard3} and definitions (\ref{delta}), (\ref{deltaretrtilde}) and (\ref{defpqrs}) imply naturally that $1\leq|r|,|s|\leq n-1$. We will show that the sign of the product $rs$ allows to distinguish among simple and complex quadrilaterals as shown in Lemma \ref{lemma2.5}. 
\begin{remark}
	Let us note that $r=0$ (resp. $s=0$) if and only if $z_j=z_k$ (resp. $z_i=z_\ell$), which yields $\sharp(\{i,j,k,\ell\})=3$. 
\end{remark}
\begin{remark}
	The equality $r=s$, which is equivalent to $p+q+2r=n$, indicates naturally that the two straight lines $\mathcal{D}_{i,j}$ and $\mathcal{D}_{k,\ell}$ are parallel and do not generate any intersection point. Of course, as shown in Lemma \ref{lemma2.5}, this situation will never occur when we deal with complex quadrilaterals since $rs<0$.	
\end{remark}
\begin{lemma}
	Let us consider an admissible quadruplet $(i,j,k,\ell)$ and its associated complex quadrilateral $\mathcal{Q}=(z_i,z_j,z_k,z_\ell)$. Then
	\begin{itemize}
		\item $r=\delta(j,k)>0$ if and only if $\delta(j,k)\leq \left[\frac{n-1}{2}\right]$ and $r=\delta(j,k)-n<0$ otherwise.
		\item When $n$ is even and $\delta(j,k)=\frac{n}{2}$, we may choose indifferently $r=\frac{n}{2}$ or $r=-\frac{n}{2}$.
		\item $\mathcal{Q}$ is simple (resp. complex) if and only if $r,s>0$ (resp. $rs<0$).
	\end{itemize}  	
	\label{lemma2.5}
\end{lemma}
\begin{proof}
	Regarding the definition (\ref{deltaretrtilde}) of $\displaystyle r=\Delta(j,k)$, we get 
	\begin{center}
		$r=\delta(j,k)>0\Leftrightarrow |\delta(j,k)|<|-\delta(k,j)|=|\delta(j,k)-n|\Leftrightarrow 0<\delta(j,k)<n-\delta(j,k)\Leftrightarrow 0<\delta(j,k)<\frac{n}{2}$.
	\end{center}
	Similarly, $r=-\delta(k,j)<0 \Leftrightarrow \delta(j,k)>\frac{n}{2}$. Since $\left[\frac{n-1}{2}\right]$ equals to $\frac{n}{2}-1$ when $n$ is even and $\frac{n-1}{2}$ otherwise, we get the first result. The second point is obvious since when $n$ is even and $\delta(j,k)=\frac{n}{2}$, $\delta(k,j)=\frac{n}{2}$ which means that $r=\frac{n}{2}$ or $r=-\frac{n}{2}$. In order to prove the third point, we use Proposition \ref{prop2.1} and Definitions \ref{defadm} and \ref{ddefpqrs}. When $\mathcal{Q}$ is simple, $r$ and $s$ respectively equal to $\delta(j,k)$ and $\delta(\ell,i)$, are both necessarily positive. When $\mathcal{Q}$ is complex, two cases may occur. If $|\delta(j,k)|>|-\delta(k,j)|$ then $r=-\delta(k,j)<0$. Let us assume that $s=n-(p+q+r)<0$ thus by using (\ref{longueurcomplexe}), which rewrites as $p+(n+r)+q+\delta(\ell,i)=2n$, we get $\delta(\ell,i)<0$ which is impossible and we deduce $rs<0$. The same conclusion holds when  $|\delta(j,k)|<|-\delta(k,j)|$ since in this case, $r=\delta(j,k)>0$ and by assuming that $s>0$, $2n=p+q+r+\delta(\ell,i)<n+\delta(\ell,i)<2n$ which is contradictory.
\end{proof}
The third point of Lemma \ref{lemma2.5} could also have been proven using the following remark.
\begin{remark} Let $(i,j,k,\ell)$ denote an admissible quadruplet and $\mathcal{Q}=(z_i,z_j,z_k,z_\ell)$ its associated quadrilateral.
	\begin{itemize}
		\item When $\mathcal{Q}$ is simple, $s=\delta(\ell,i)$.
		\item When $\mathcal{Q}$ is complex, $s=\left\{\begin{array}{ccl}-\delta(i,\ell)=\delta(\ell,i)-n & \mbox{ if } & r=\delta(j,k)\\
			\delta(\ell,i) & \mbox{ if } & r=-\delta(k,j)=\delta(j,k)-n\end{array}\right..$
	\end{itemize}
	\label{remdelta}
\end{remark}
These results follow immediately from Lemma \ref{lemma2.5}.
\begin{example}
	We consider the case of the clique-arrangement with $n=8$ as well as the simple quadrilateral $(z_2,z_4,z_7,z_1)$ and the complex one $(z_1,z_4,z_2,z_7)$. 
	In the first case, we have $(p,q,r)=(2,2,1)$, while in the second case, $(p,q,r)=(3,3,-1)$, $s$ being equal to 3 in both cases. If we consider the complex quadrilateral $(z_2,z_1,z_4,z_7)$, $\sharp\left(\curvearc{z_2z_1}\cap\{z_4,z_7\}\right)=2$ and $(p,q,r)=(7,3,3)$ while for $(z_1,z_2,z_7,z_4)$, $\sharp\left(\curvearc{z_7z_4}\cap\{z_2,z_1\}\right)=2$ and $(p,q,r)=(1,5,-3)$.
\end{example}
As mentioned previously, to any admissible quadruplet $(i,j,k,\ell)$ (and then to its associated quadrilateral $\mathcal{Q}$) corresponds only one triplet $(p,q,r)$ and only one intersection point $z_{i,j,k,\ell}$. Nevertheless, there is no 1-to-1 correspondance between $z_{i,j,k,\ell}$ and $(p,q,r)$. Indeed, a triplet $(p,q,r)$, unlike an admissible quadruplet $(i,j,k,\ell)$, characterizes a specific shape of a quadrilateral (and then a specific orbit) as well as a set of intersection points, but not a quadrilateral or an intersection point in particular.\\ 
%

\section{Orbits of clique-arrangements}

The circular orbits generated by the clique-arrangement may be defined by their distance from the origin and which has been characterized in \cite{RS2} help to the square modulus function $J_n$. 
For various integers $n,i,j,k,\ell $, such that $i+j-(k+\ell)$ is not a multiple of $n$ we set 
\begin{equation}
	J_n(i,j,k,\ell )=\frac{\cos ^2(ij_{-})+\cos ^2(k\ell _{-})-2\cos (ij_{-})\cos (k\ell _{-})\cos (ij_{+}-k\ell _{+})}{\sin ^2(ij_{+}-k\ell _{+})}. 
	\label{Jn}
\end{equation}
When $k+\ell \not\equiv 0\mbox{ mod }n$ and $i+j\equiv 0\mbox{ mod }n$, we set
\begin{equation*}
	\widetilde J_n(i,j,k,\ell )=\frac{\cos ^2(ii_{+})+\cos ^2(k\ell _{-})-2\cos (ii_{+})\cos (k\ell _{-})\cos (k\ell _{+})}{\sin ^2(k\ell _{+})}.
	\label{Jnprime}
\end{equation*}
The range of $\widetilde{J_n}$ is a subset of the range of $J_n$ which are sets of algebraic numbers and nothing but the squares of the radii of the orbits containing the vertices of $\mathcal{K}_n$. We have $\Im(J_n)\subset [0,\left(\sin\left(\frac{\pi}{n}\right)\right)^{-2}]$.\\
The two previous formulas look strangely like to some normalized Al-Kashi's formulas combining angles and side lengths characterizing the quadrilateral $(z_i,z_j,z_k,z_\ell)$. Let us note at this point the important role played by the three integers $i-j$, $k-\ell$ and $i+j-(k+\ell)$ in the formula (\ref{Jn}).\\

First we note that the problem of computing the modulii of the orbits does not necessarily require a clique-arrangement. Indeed, if we consider four points $z_A=R\exp(\alpha_A I)$, $z_B=R\exp(\alpha_B I)$, $z_C=R\exp(\alpha_C I)$ and $z_D=R\exp(\alpha_D I)$ on a circle centered at the origin and of radius $R$, they form a cyclic or inscribed quadrilateral. The six segments $[z_A,z_B]$, $[z_B,z_C]$, $[z_C,z_D]$, $[z_D,z_A]$, $[z_B,z_D]$ and $[z_A,z_C]$ endlessy extended respectively to $\mathcal{D}_{A,B}$, $\mathcal{D}_{B,C}$, $\mathcal{D}_{C,D}$, $\mathcal{D}_{D,A}$, $\mathcal{D}_{B,D}$ and $\mathcal{D}_{A,C}$ generate two external and one internal intersection points, respectively $z_{A,B,C,D}$, $z_{A,D,B,C}$ and $z_{A,C,B,D}$, provided no straight line is parallel to another (see Figure \ref{fig2}).
\begin{figure}[!ht]
	\begin{center}
		\includegraphics[width=10cm,height=7cm]{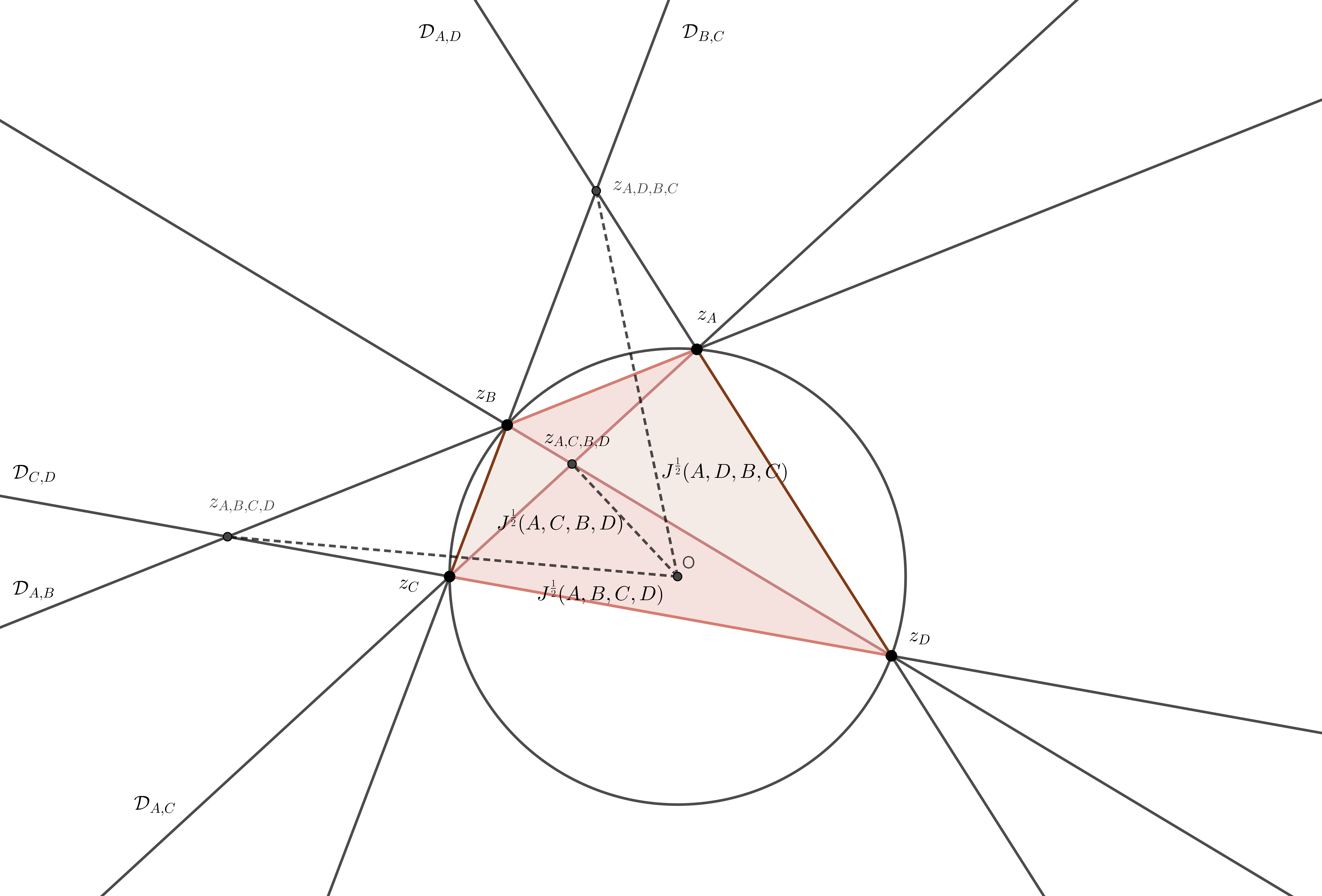}
	\end{center}
	\caption{\centering The 3 cyclic quadrilaterals generated from 4 vertices $z_A,z_B,z_C,z_D$ and their associated 3 intersection points.}
	\label{fig2}
\end{figure}
\noindent Clearly, the way the four points are distributed on the circle has a crucial importance on the location of the intersection point and its distance from the origin of the unit circle. Regarding $J(A,B,C,D)$ for example, which is defined as a natural extension of (\ref{Jn}) to four vertices $z_A,z_B,z_C,z_D$ lying on a circle of radius $R$ centered at $0$ but not necessarily located on preset positions as following
\begin{equation}
	\footnotesize{J(A,B,C,D)=R\frac{\cos ^2(\alpha_A-\alpha_B)+\cos ^2(\alpha_C-\alpha_D)-2\cos (\alpha_A-\alpha_B)\cos (\alpha_C-\alpha_D)\cos (\alpha_A+\alpha_B-(\alpha_C+\alpha_D))}{\sin ^2(\alpha_A+\alpha_B-(\alpha_C+\alpha_D))},}
	\label{Jncont}
\end{equation}
the segments $[z_A,z_B]$ and $[z_C,z_D]$, or $[z_B,z_C]$ and $[z_A,z_D]$, need to be furthest from each other and nearly parallel in order to make the value of $J$ as large as possible. On the other hand, the segments $[z_A,z_B]$ and $[z_C,z_D]$ need to intersect inside the circle and to be (almost) diameters in order to make the value of $J$ (almost) equal to 0. The formula (\ref{Jncont}) which provides the distances to the origin of the two intersection points generated in this way, seems to be new.\\

Next we come back to the formula occuring in \cite{RS1} and \cite{RS2} for $J_n(i,j,k,\ell)$ and we express it help to the three parameters $p,q,r$.

\begin{theorem}
	With the previous notation, the formula (\ref{Jn}) for $J_n$ occuring in \cite{RS1} and \cite{RS2} redefines as
	\begin{equation}
		J_n(p,q,r)=\frac{\cos ^2(p\frac{\pi}{n})+\cos ^2(q\frac{\pi}{n})-2\cos (p\frac{\pi}{n})\cos (q\frac{\pi}{n})\cos ((p+q+2r)\frac{\pi}{n})}{\sin ^2((p+q+2r)\frac{\pi}{n})},
		\label{Jnmod}
	\end{equation}
	which may rewrite matricially as
	\[J_n(p,q,r)=\frac{1}{\det(A(p,q,r))}v(p,q)^TA(p,q,r)v(p,q)\]
	where $A(p,q,r)=\begin{pmatrix} 1 & -\cos((p+q+2r)\frac{\pi}{n})\\-\cos((p+q+2r)\frac{\pi}{n}) & 1\end{pmatrix}$ and $v(p,q)=\begin{pmatrix}\cos(p\frac{\pi}{n})\\\cos(q\frac{\pi}{n})\end{pmatrix}$.
	\label{thmJn}
\end{theorem}
\begin{proof}
As a consequence of Lemma \ref{lemma2.3}, and if $(a,b)$ denotes the single pair of indices such that $z_a$ and $z_b$ are on either sides of $z_0$, the different components occuring in (\ref{Jn}) when $\mathcal{Q}$ is a simple quadrilateral may rewrite as follows:
\begin{table}[!ht]
	\centering
	\begin{tabular}{|c||c|c|c|c|}
		\hline
		$(a,b)$ & $\cos(ij_{-})$ & $\cos(k\ell_{-})$ & $ij_+-k\ell_+$ & $\cos(ij_+-k\ell_+)$\\
		\hline
		\hline
		$(i,j)$ & $-\cos(\delta(i,j)\frac{\pi}{n})$ & $\cos(\delta(k,\ell)\frac{\pi}{n})$ & $(\delta(\ell,i)-\delta(j,k))\frac{\pi}{n}$ & $\cos((\delta(j,k)-\delta(\ell,i))\frac{\pi}{n})$\\		
		\hline
		$(j,k)$ & $\cos(\delta(i,j)\frac{\pi}{n})$ & $\cos(\delta(k,\ell)\frac{\pi}{n})$ & $(n+\delta(\ell,i)-\delta(j,k))\frac{\pi}{n}$ & $-\cos((\delta(j,k)-\delta(\ell,i))\frac{\pi}{n})$\\		
		\hline
		$(k,\ell)$ & $\cos(\delta(i,j)\frac{\pi}{n})$ & $-\cos(\delta(k,\ell)\frac{\pi}{n})$ & $(\delta(\ell,i)-\delta(j,k))\frac{\pi}{n}$ & $\cos((\delta(j,k)-\delta(\ell,i))\frac{\pi}{n})$\\
		\hline
		$(\ell,i)$ & $\cos(\delta(i,j)\frac{\pi}{n})$ & $\cos(\delta(k,\ell)\frac{\pi}{n})$ & $(-n+\delta(\ell,i)-\delta(j,k))\frac{\pi}{n}$ & $-\cos((\delta(j,k)-\delta(\ell,i))\frac{\pi}{n})$\\		
		\hline		
	\end{tabular}
	\caption{Components of $J_n$ with respect to (w.r.t.) $\delta$ when $\mathcal{Q}$ is simple}
\end{table}

\noindent When $\mathcal{Q}$ is a complex quadrilateral, we have $\sharp(\mathcal{S})=2$. Then, if $(a_1,b_1)$ and $(a_2,b_2)$ denote the two pairs of indices such that $z_{a_1}$ and $z_{b_1}$ as well as $z_{a_2}$ and $z_{b_2}$ are on either sides of $z_0$, the following table holds
\begin{table}[!ht]
	\centering
	\begin{tabular}{|c|c||c|c|c|c|}
		\hline
		$(a_1,b_1)$ & $(a_2,b_2)$ & $\cos(ij_{-})$ & $\cos(k\ell_{-})$ & $ij_+-k\ell_+$ & $\cos(ij_+-k\ell_+)$\\
		\hline
		\hline
		$(i,j)$ & $(j,k)$ & $-\cos(\delta(i,j)\frac{\pi}{n})$ & $\cos(\delta(k,\ell)\frac{\pi}{n})$ & $(n+\delta(\ell,i)-\delta(j,k))\frac{\pi}{n}$ & $-\cos((\delta(j,k)-\delta(\ell,i))\frac{\pi}{n})$\\
		\hline
		$(i,j)$ & $(k,\ell)$ & $-\cos(\delta(i,j)\frac{\pi}{n})$ & $-\cos(\delta(k,\ell)\frac{\pi}{n})$ & $(\delta(\ell,i)-\delta(j,k))\frac{\pi}{n}$ & $\cos((\delta(j,k)-\delta(\ell,i))\frac{\pi}{n})$\\
		\hline
		$(i,j)$ & $(\ell,i)$ & $-\cos(\delta(i,j)\frac{\pi}{n})$ & $\cos(\delta(k,\ell)\frac{\pi}{n})$ & $(n+\delta(\ell,i)-\delta(j,k))\frac{\pi}{n}$ &  $-\cos((\delta(j,k)-\delta(\ell,i))\frac{\pi}{n})$ \\		
		\hline
		$(j,k)$ & $(k,\ell)$ & $\cos(\delta(i,j)\frac{\pi}{n})$ & $-\cos(\delta(k,\ell)\frac{\pi}{n})$ & $(n+\delta(\ell,i)-\delta(j,k))\frac{\pi}{n}$ & $-\cos((\delta(j,k)-\delta(\ell,i))\frac{\pi}{n})$\\
		\hline
		$(j,k)$ & $(\ell,i)$ & $\cos(\delta(i,j)\frac{\pi}{n})$ & $\cos(\delta(k,\ell)\frac{\pi}{n})$ & $(\delta(\ell,i)-\delta(j,k))\frac{\pi}{n}$ & $\cos((\delta(j,k)-\delta(\ell,i))\frac{\pi}{n})$\\			
		\hline
		$(k,\ell)$ & $(\ell,i)$ & $\cos(\delta(i,j)\frac{\pi}{n})$ & $-\cos(\delta(k,\ell)\frac{\pi}{n})$ & $(-n+\delta(\ell,i)-\delta(j,k))\frac{\pi}{n}$ & $-\cos((\delta(j,k)-\delta(\ell,i))\frac{\pi}{n})$\\
		\hline		
	\end{tabular}
	\caption{Components of $J_n$ w.r.t. $\delta$ when $\mathcal{Q}$ is complex}
\end{table}
Therefore, in any situation, we may write that
\begin{itemize}
	\item[.] $\cos ^2(ij_{-})+\cos ^2(k\ell _{-})=\cos^2(p\frac{\pi}{n})+\cos^2(q\frac{\pi}{n})$,  
	\item[.] $\sin ^2(ij_{+}-k\ell _{+})=\sin^2((\delta(j,k)-\delta(\ell,i))\frac{\pi}{n})$ and 
	\item[.] $\cos (ij_{-})\cos (k\ell _{-})\cos (ij_{+}-k\ell _{+})=(-1)^{\sharp(\mathcal{S})}\cos(p\frac{\pi}{n})\cos(q\frac{\pi}{n})\cos((\delta(j,k)-\delta(\ell,i))\frac{\pi}{n})$.
	\end{itemize} 
When $\mathcal{Q}$ is simple, 
\[\cos((\delta(j,k)-\delta(\ell,i))\frac{\pi}{n})=\cos((r-s)\frac{\pi}{n})=\cos((p+q+2r-n)\frac{\pi}{n})=-\cos((p+q+2r)\frac{\pi}{n}),\]
while in the case when $\mathcal{Q}$ is complex, using Remark \ref{remdelta}, $r-s=\pm n+\delta(j,k)-\delta(\ell,i)$ and consequently,
\[-\cos((\delta(j,k)-\delta(\ell,i))\frac{\pi}{n})=\cos((r-s)\frac{\pi}{n})=-\cos((p+q+2r)\frac{\pi}{n}).\]
Therefore, we get in any case, $\cos(ij_{-})\cos(k\ell_{-})\cos(ij_+-k\ell_+)=\cos(p\frac{\pi}{n})\cos(q\frac{\pi}{n})\cos((p+q+2r)\frac{\pi}{n})$.
\end{proof}
\begin{example} \rm
	As an example, we consider the case of the clique-arrangement with $n=20$. The different values of $\sqrt{J_n}$ w.r.t. the $410$ external orbits and the $334$ internal orbits are given in Figures \ref{figureJnext} and \ref{figureJnint} respectively. We note that the behaviour of $\sqrt{J_n}$ is different outside and inside the unit circle as well as the number of orbits. 
	\begin{figure}[!ht]
		\begin{minipage}[c]{.46\linewidth}
			\centering\includegraphics[width=5.5cm,height=3.7cm]{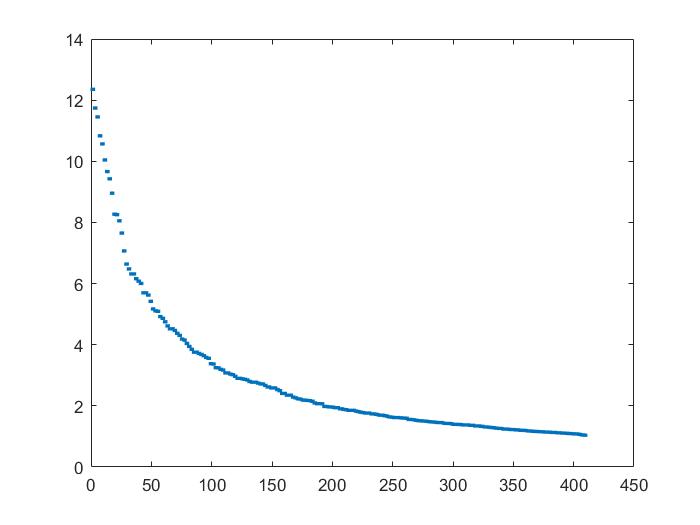}
			\caption{Values of $\sqrt{J_{20}}$ outside the unit circle}
			\label{figureJnext}
		\end{minipage} \hfill
		\begin{minipage}[c]{.46\linewidth}
			\centering\includegraphics[width=5.5cm,height=3.7cm]{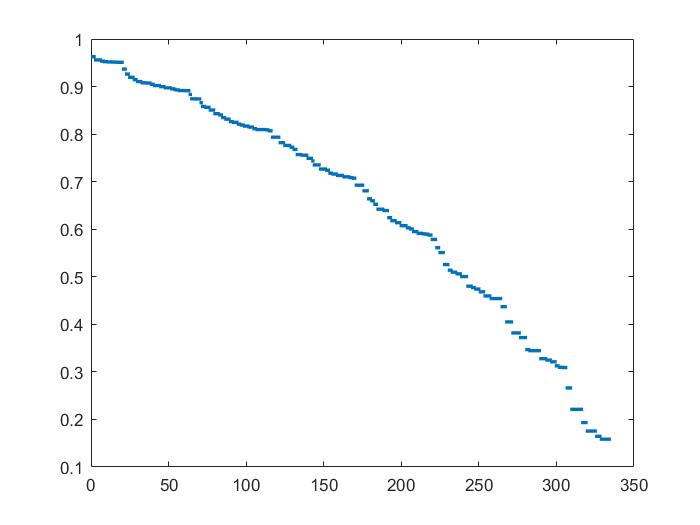}
			\caption{Values of $\sqrt{J_{20}}$ inside the unit circle}
			\label{figureJnint}
		\end{minipage}
	\end{figure}
\end{example}
\begin{proposition}
	For all triplet $(p,q,r)$ associated to any inscribed quadrangle, we have 
	\begin{equation}
		J_n(p,q,r)=J_n(q,p,r)\qquad\mbox{and}\qquad J_n(p,q,r)=J_n(p,q,s).
		\label{propJn}
	\end{equation}
\label{prop3.1}
\end{proposition}
\begin{proof}
	We notice first that because of the symmetry of $J_n$ with respect to its two first components, $J_n(p,q,r)=J_n(q,p,r)$. Next, because $s=n-(p+q+r)$, we have $\cos((p+q+2s)\frac{\pi}{n})=\cos((p+q+2(n-(p+q+r)))\frac{\pi}{n})=\cos((-p-q-2r+2n)\frac{\pi}{n})$, and consequently, $J_n(p,q,s)=J_n(p,q,r)$. 
\end{proof}
Some other triplets may ensure the invariance of $J_n(p,q,r)$ as for example $(q,p,s)$, $(n-p,q,p+r)$ and $(n-p,q,-q-r)$, and all of them are deduced from involutions of the set of triplets satisfying (\ref{defpqrs}). Nevertheless, they do not characterize the same intersection points as shown in the following example.
\begin{example}
	We consider the case $n=9$ and the triplets $(p,q,r)=(3,2,-1)$ and $(p,q,r)=(2,3,-1)$ which generate respectively the internal points $\{P_i\}_{i\in\{0,8\}}$ and $\{Q_i\}_{i\in\{0,8\}}$. As Figure \ref{fig_pqr_andqpr} shows, they are all located on the circle of radius $J_9(3,2,-1)^\frac12=J_9(2,3,-1)^\frac12$ but they are all distinct.
	\begin{figure}[!ht]
		\begin{center}
			\includegraphics[width=6.5cm,height=5cm]{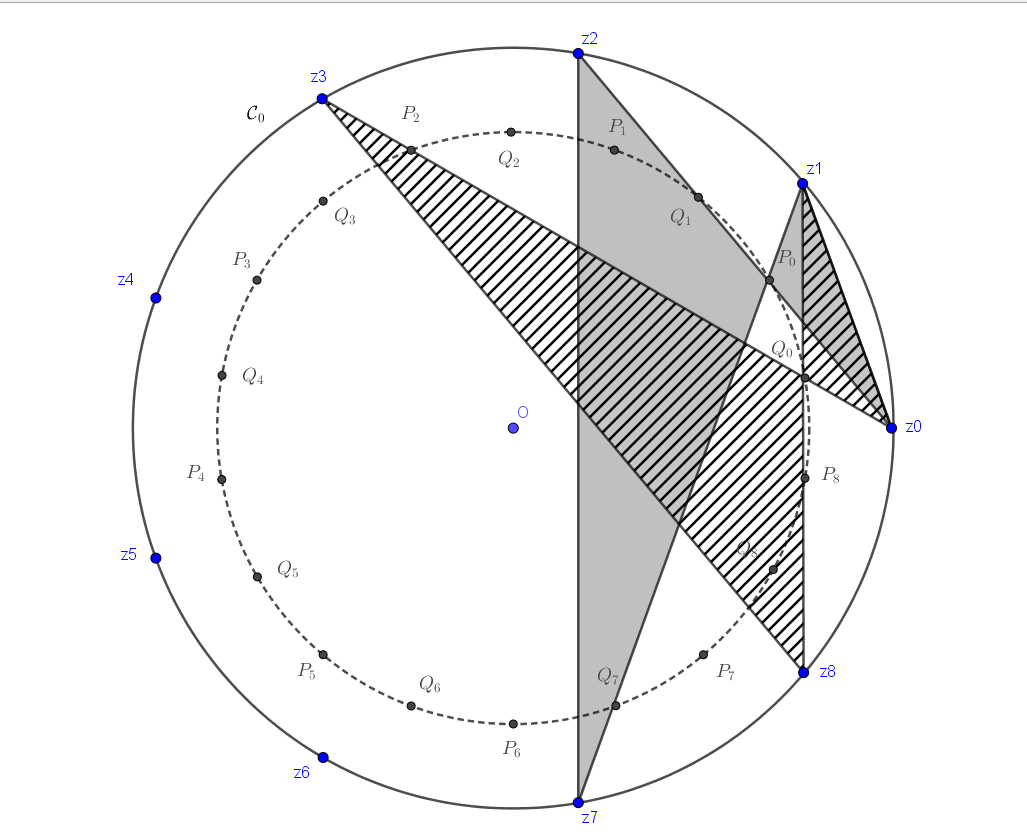}
		\end{center}
	\caption{Distinct cocyclic vertices of $\mathcal{K}_9$ generated from inequivalent triplets.}
	\label{fig_pqr_andqpr}
	\end{figure}
\end{example}
\begin{definition}
	Two triplets $(p_1,q_1,r_1)$ and $(p_2,q_2,r_2)$ are said equivalent if and only if the two sets of cardinality $n$ consisting of intersection points associated to quadrilaterals constructed from $(p_1,q_1,r_1)$ and $(p_2,q_2,r_2)$ are congruent modulo the group of order $n$ of central rotations of angle $\frac{2k\pi}{n}$, $k\in\{0,\ldots,n-1\}$.
	\label{eqtriplets}
\end{definition}
\begin{remark}
	A necessary condition for $(p_1,q_1,r_1)$ and $(p_2,q_2,r_2)$ to be equivalent is that $J_n(p_1,q_1,r_1)=J_n(p_2,q_2,r_2)$. The integers $p_1,q_1,r_1,p_2,q_2$ being given, solving $J_n(p_1,q_1,r_1)=J_n(p_2,q_2,r_2)$ allows to determine the admissible values of $r_2$ such that $(p_1,q_1,r_1)\equiv(p_2,q_2,r_2)$. 
\end{remark}

It is therefore important to develop a way, for a given radius $J_n^\frac12$, to distinguish the points which are associated to some specific triplets $(p,q,r)$ to those which are not. This is the aim of the next section.
	
\section{About the arc length between two points lying on the same orbit}

Let us assume that $P_1=z_{i_1,j_1,k_1,\ell_1}=(x_1,y_1)$ and $P_2=z_{i_2,j_2,k_2,\ell_2}=(x_2,y_2)$ are two intersection points generated respectively by the two triplets $(p_1,q_1,r_1)$ and $(p_2,q_2,r_2)$ such that $J_n(p_1,q_1,r_1)=J_n(p_2,q_2,r_2)=J$. Since the two points $P_1$ and $P_2$ lie on the same circle whose center point is $(0,0)$ and radius equals to $J^\frac12$, the circular arc length $d_{P_1,P_2}$ between these two points is defined as follows
\begin{equation}
	d_{P_1,P_2}=J^\frac12\cos^{-1}\left(\frac{x_1x_2+y_1y_2}{\sqrt{(x_1^2+y_1^2)(x_2^2+y_2^2)}}\right)=J^{\frac12}\cos^{-1}\left(\frac{x_1x_2+y_1y_2}{J}\right).
	\label{distang}
\end{equation}
Let us remind (see \cite{RS2}) that the coordinates of $z_{i,j,k,\ell}=\mathcal{D}_{i,j}\cap\mathcal{D}_{k,\ell}$ are given by
\small{\begin{eqnarray}
		x &=&\frac{\sin (ij_{+})\cos (k\ell _{-})-\cos (ij_{-})\sin (k\ell_{+})}{\sin (ij_{+})\cos (k\ell _{+})-\cos (ij_{+})\sin (k\ell _{+})}=
		\frac{\sin (ij_{+})\cos (k\ell _{-})-\cos (ij_{-})\sin (k\ell _{+})}{\sin (\theta )}, \label{xijkl1}\nonumber\\
		y &=&\frac{\cos (ij_{-})\cos (k\ell _{+})-\cos (ij_{+})\cos (k\ell_{-})}{\sin (ij_{+})\cos (k\ell _{+})-\cos (ij_{+})\sin (k\ell _{+})}=
		\frac{\cos (ij_{-})\cos (k\ell _{+})-\cos (ij_{+})\cos (k\ell _{-})}{\sin (\theta )},\label{yijkl1}\nonumber
\end{eqnarray}}
where $\theta =ij_{+}-k\ell _{+}=\left(i+j-(k+\ell)\right)\frac\pi n$. Therefore, if $\theta_1 =i_1j_{1+}-k_1\ell _{1+}$ and $\theta_2=i_2j_{2+}-k_2\ell _{2+}$, 
\begin{equation} 
	\begin{split}
	d_{P_1,P_2}=J^\frac12\cos^{-1}\left(\frac{1}{J\sin(\theta_1)\sin(\theta_2)}\left[\cos(k_2\ell_{2-})\left(\cos(k_1\ell_{1-})\cos((i_1j_{1+})-(i_2j_{2+}))-\cos(i_1j_{1-})\cos((k_1\ell_{1+})-(i_2j_{2+}))\right)+\right.\right.\\
	\left.\cos(i_2j_{2-})(\cos(i_1j_{1-})\cos((k_1\ell_{1+})-(k_2\ell_{2+}))-\cos(k_1\ell_{1-})\cos((i_1j_{1+})-(k_2\ell_{2+})))\right]\biggr).
	\end{split}
\label{dang}
\end{equation}
\[=J^\frac12\cos^{-1}\left(\frac{1}{J\sin(\theta_1)\sin(\theta_2)}\left[\begin{pmatrix}\cos(p_2\frac{\pi}{n}) & \cos(q_2\frac{\pi}{n})\end{pmatrix}A_{P_1,P_2}\begin{pmatrix}\cos(p_1\frac{\pi}{n}) \\ \cos(q_1\frac{\pi}{n})\end{pmatrix}\right]\right)\]
where $A_{P_1,P_2}=\begin{pmatrix}\cos(\alpha_{k,\ell}) & -\cos(\alpha_{k,\ell}+\theta_1)\\-\cos(\alpha_{k,\ell}-\theta_2) & \cos(\alpha_{k,\ell}+\theta_1-\theta_2)\end{pmatrix}$ and $\alpha_{k,\ell}=k_1\ell_{1+}-k_2\ell_{2+}$. 

Without loss of generality, we may assume that $i_1=0$ and $i_2=0$ then by using Lemma \ref{lemma2.4}, $(i_1,j_1,k_1,\ell_1)=(0,p_1,p_1+r_1,p_1+q_1+r_1)$ and $(i_2,j_2,k_2,\ell_2)=(0,p_2,p_2+r_2,p_2+q_2+r_2)$. Hence,  $\theta_1=(p_1-(2p_1+q_1+2r_1))\frac{\pi}{n}=-(p_1+q_1+2r_1)\frac{\pi}{n}$, $\theta_2=(p_2-(2p_2+q_2+2r_2))\frac{\pi}{n}=-(p_2+q_2+2r_2)\frac{\pi}{n}$ and $\alpha_{k,\ell}=(2p_1+q_1+2r_1)\frac{\pi}{n}-(2p_2+q_2+2r_2)\frac{\pi}{n}=(2(p_1-p_2)+(q_1-q_2)+2(r_1-r_2))\frac{\pi}{n}$. Thus, (\ref{dang}) rewrites as
\begin{equation} 
d_{P_1,P_2}=J^\frac12\cos^{-1}\left(\frac{1}{J\sin((p_1+q_1+2r_1)\frac{\pi}{n})\sin((p_2+q_2+2r_2)\frac{\pi}{n})}\left[\begin{pmatrix}\cos(p_2\frac{\pi}{n}) & \cos(q_2\frac{\pi}{n})\end{pmatrix}A_{P_1,P_2}\begin{pmatrix}\cos(p_1\frac{\pi}{n}) \\ \cos(q_1\frac{\pi}{n})\end{pmatrix}\right]\right)
\label{dP1P2}
\end{equation} 
where 
\begin{equation}
	A_{P_1,P_2}=\begin{pmatrix}\cos((2(p_1-p_2)+(q_1-q_2)+2(r_1-r_2))\frac{\pi}{n}) & -\cos((p_1-(2p_2+q_2+2r_2))\frac{\pi}{n})\\
		-\cos(((2p_1+q_1+2r_1)-p_2)\frac{\pi}{n}) & \cos((p_1-p_2)\frac{\pi}{n})\end{pmatrix}.
	\label{AP1P2}
\end{equation}	
\begin{remark}
	With notation of Theorem \ref{thmJn}, we note that $A_{P_1,P_1}=A(p_1,q_1,r_1)$ and consequently, 
	\begin{center}
		$d_{P_1,P_1}=J^\frac12\cos^{-1}\left(\frac{1}{J\det(A(p_1,q_1,r_1))}[v(p_1,q_1)^TA(p_1,q_1,r_1)v(p_1,q_1)]\right)=0$.
	\end{center}
\end{remark}
\begin{proposition}
	Let $J^\frac12$ be the radius of some orbit of the clique-arrangement $\mathcal{K}_n$.
	\begin{enumerate}
		\item The triplets $(p_1,q_1,r_1)$ and $(p_2,q_2,r_2)$ are equivalent if and only if $J_n(p_1,q_1,r_1)=J_n(p_2,q_2,r_2)=J$ and $d_{P_1,P_2}J^{-\frac12}$ is a multiple, say $\rho$ (defined modulo $n$), of $\frac{2\pi}{n}$.
		\item The cardinality of the orbit $J_n(p,q,r)=J$ is equal to $n\nu$ where $\nu$ is the number of unequal angular distances modulo $J^\frac12\frac{2\pi}{n}$
	\end{enumerate}
	\label{prop4.1}
\end{proposition}
\begin{proof}
	\begin{enumerate}
		\item The triplet $(p,q,r)$ spans exactly $n$ distinct quadrangles being rotated from each other around the origin through an angle $\frac{2\pi}{n}$. The suitable opposite sides (imposed by $p$ and $q$) generate $n$ distinct points regularly distributed on the circular orbit of origin $0$ and of radius $J_n(p,q,r)^\frac12$. Then in order to make the sets of the $n$ points generated by $(p_1,q_1,r_1)$ and by $(p_2,q_2,r_2)$, say $S_1$ and $S_2$ respectively, coincide, a necessary condition is that these points lie on the same circular orbit of radius $J_n(p_1,q_1,r_1)^\frac12=J_n(p_2,q_2,r_2)^\frac12$. This condition being verified, $\sharp(S_1\cup S_2)=n$ if and only if the arc length $d_{P_1,P_2}$ between two points $P_1$ and $P_2$ lying respectively in $S_1$ and $S_2$ is a multiple of $J^\frac12\frac{2\pi}{n}$. 
		\item This point is a direct consequence of the previous result.
	\end{enumerate}
\end{proof}
Let us assume that $(p_1,q_1,r_1)\equiv (p_2,q_2,r_2)$ then the intersection points generated by using $(i_1,p_1,q_1,r_1)$ and $(i_2,p_2,q_2,r_2)$ are respectively
\begin{center}
	$z_{i_1,i_1+p_1,i_1+p_1+r_1,i_1+p_1+q_1+r_1}$ and $z_{i_2,i_2+p_2,i_2+p_2+r_2,i_2+p_2+q_2+r_2}$.
\end{center} 
Obviously, a sufficient condition for such two points to be equal is that the corresponding indices are equal, i.e. $i_2=i_1\pm\rho$. In order to determine if $i_2=i_1+\rho$ or $i_2=i_1-\rho$, it suffices to check one of the two equalities
\begin{center} 
	$z_{i_1+\rho,i_1+\rho+p_2,i_1+\rho+p_2+r_2,i_1+\rho+p_2+q_2+r_2}=z_{i_1,i_1+p_1,i_1+p_1+r_1,i_1+p_1+q_1+r_1}$,\\
	$z_{i_1-\rho,i_1-\rho+p_2,i_1-\rho+p_2+r_2,i_1-\rho+p_2+q_2+r_2}=z_{i_1,i_1+p_1,i_1+p_1+r_1,i_1+p_1+q_1+r_1}$.
\end{center} 
Depending on the case, we shall note
\begin{equation}
	 (p_1,q_1,r_1)\equiv(p_2,q_2,r_2)~~[\pm\rho].
	 \label{notation_rho}
\end{equation} 

\noindent Naturally, the following results hold
\begin{proposition}$~$\vspace{-0.5cm}\\
\begin{enumerate}
	\item $(p_1,q_1,r_1)\equiv(p_2,q_2,r_2)~~[\rho]\Leftrightarrow (p_2,q_2,r_2)\equiv(p_1,q_1,r_1)~~[-\rho]$.
	\item If $(p_1,q_1,r_1)\equiv(p_2,q_2,r_2)~~[\rho_1]$ and $(p_2,q_2,r_2)\equiv(p_3,q_3,r_3)~~[\rho_2]$, then $(p_1,q_1,r_1)\equiv(p_3,q_3,r_3)~~[\rho_1+\rho_2]$.
\end{enumerate}
\label{prop4.2}
\end{proposition}
The proof of these results is obvious.\\

We know that the number of distinct pairs $\{i,j\}$ characterizing a straight line $\mathcal{D}_{i,j}$ passing through a specific intersection point defines its multiplicity. If $(p_1,q_1,r_1)\equiv(p_2,q_2,r_2)~~[\rho]$, the point $z_{i_1,i_1+p_1,i_1+p_1+r_1,i_1+p_1+q_1+r_1}$, $i_1\in\{0,\ldots,n-1\}$ is common to the four straight lines 
\begin{center}
$\mathcal{D}_{i_1,i_1+p_1}$, $\mathcal{D}_{i_1+p_1+r_1,i_1+p_1+q_1+r_1}$, $\mathcal{D}_{i_1+\rho,i_1+\rho+p_2}$ and $\mathcal{D}_{i_1+\rho+p_2+r_2,i_1+\rho+p_2+q_2+r_2}$.
\end{center}
So, in order to determine the multiplicity of each of the $n$ points defined by $(p_1,q_1,r_1)$ or $(p_2,q_2,r_2)$, it suffices to enumerate the distinct pairs of indices characterizing the straight lines among the $4=2^2$ possible choices. When $k$ triplets are equivalent, we may theoretically have a maximum of $2^k$ distinct pairs. Nevertheless, Poonen and Rubinstein \cite{PR} have proved that the number of $\ell$-tuples of diagonals which meet at a point inside the $n$-gon other than the center never exceed 7 what we will observe in our numerical experiments. Therefore, no matter the number of equivalent triplets, the number of distinct pairs of indices characterizing the straight lines intersecting at an internal intersection point will not exceed 7.   

\begin{proposition}
For all integers $p,q,r$ characterized through Definition \ref{ddefpqrs}, we have
		\begin{equation}
			(p,q,r)\equiv(q,p,\underbrace{n-p-q-r}_{s})\qquad [p+r].
			\label{eqpqr0}
		\end{equation}
	\label{prop4.3}
\end{proposition}

\begin{proof}
In order to use (\ref{dP1P2}) and (\ref{AP1P2}) we consider $(p_1,q_1,r_1)=(p,q,r)$ and $(p_2,q_2,r_2)=(q,p,r)$ and we compute
	\begin{itemize}
		\item[.] $\sin((p_1+q_1+2r_1)\frac{\pi}{n})=\sin((p+q+2r)\frac{\pi}{n})$,
		\item[.] $\sin((p_2+q_2+2r_2)\frac{\pi}{n})=-\sin((p+q+2r)\frac{\pi}{n})$,
		\item[.] $\cos((p_1-p_2)\frac{\pi}{n})=\cos((p-q)\frac{\pi}{n})$,
		\item[.] $-\cos((p_1-(2p_2+q_2+2r_2))\frac{\pi}{n})=-\cos(2(p+r)\frac{\pi}{n})$,
		\item[.] $-\cos((2p_1+q_1+2r_1)-p_2)\frac{\pi}{n})=-\cos(2(p+r)\frac{\pi}{n})$,
		\item[.] $\cos((2(p_1-p_2)+(q_1-q_2)+2(r_1-r_2))\frac{\pi}{n})=\cos((3p+q+4r)\frac{\pi}{n})$
	\end{itemize}
Then, 
\[d_{P_1,P_2}=J^\frac12\cos^{-1}\left(\frac{-1}{J\sin^2((p+q+2r)\frac{\pi}{n})}\times\right.\]
\[\left(\begin{pmatrix}\cos(q\frac{\pi}{n}) & \cos(p\frac{\pi}{n})\end{pmatrix}
\begin{pmatrix}\cos((3p+q+4r)\frac{\pi}{n}) & -\cos(2(p+r)\frac{\pi}{n})\\-\cos(2(p+r)\frac{\pi}{n}) & \cos((p-q)\frac{\pi}{n})\end{pmatrix}\begin{pmatrix}\cos(p\frac{\pi}{n})\\\cos(q\frac{\pi}{n})\end{pmatrix}\right)\]
\[=J^\frac12\cos^{-1}\left(\frac{-1}{J\sin^2((p+q+2r)\frac{\pi}{n})}\times\right.\]
\[\left.\left(2\cos\left(p\frac{\pi}{n}\right)\cos\left(q\frac{\pi}{n}\right)\cos\left(2(p+r)\frac{\pi}{n}\right)\cos\left((p+q+2r)\frac{\pi}{n}\right)-\cos\left(2(p+r)\frac{\pi}{n}\right)\left(\cos^2\left(p\frac{\pi}{n}\right)+\cos^2\left(q\frac{\pi}{n}\right)\right)\right)\right)\]		
\[=J^\frac12\cos^{-1}\left(\cos\left(2(p+r)\frac{\pi}{n}\right)\right)=J^\frac12(p+r)\frac{2\pi}{n}.\]
Help to Proposition \ref{prop4.1}, we deduce that $(p,q,r)\equiv(q,p,n-(p+q+r))$.
This result means first that the $n$ points generated by $(p,q,r)$ and $(q,p,s)$ coincide and second, that the two intersection points generated by the two quadrangles defined by $i_1=i_2=0$ are distant from each other of an angular distance $d_{P_1,P_2}$. So if we want that $P_1$ and $P_2$ (and the $2n-2$ comparable intersection points) coincide, we have to choose $i_2=i_1+(p+r)$ since  $(i_2,i_2+q,i_2+n-p-r,i_2+n-r)=(i_1+p+r,i_1+p+q+r,i_1,i_1+p)$.
\end{proof}
\begin{proposition}
	If $\mathcal{Q}$ is complex,
	\begin{alignat}{2}
		(p,q,r) & \equiv (q,n-p,-q-r) \quad & & [p+r]\label{eqpqr1}\\
		        & \equiv (n-p,n-q,p+q+r-n) \quad & & [p]\label{eqpqr2}\\
		        & \equiv (n-q,p,-p-r) \quad & & [p+q+r]\label{eqpqr3}
	\end{alignat}
\label{prop4.4}
\end{proposition}
\begin{proof}
	We know that when $\mathcal{Q}$ is simple, it is characterized by triplets $(p,q,r)$ such that $r$ and $s$ are positive. Then, when $\mathcal{Q}$ is simple, the equivalences (\ref{eqpqr1}),(\ref{eqpqr2}) and (\ref{eqpqr3}) are impossible. Indeed, if $r>0$, $-q-r<0$, $p+q+r-n<0$ and $-p-r<0$ which is incompatible with the fact that $\mathcal{Q}$ is simple.\\
	The first equivalence (\ref{eqpqr1}) may be proved in the same way as in Proposition \ref{prop4.3}. Next, if we define the mapping $T:(p,q,r)\rightarrow (q,n-p,-q-r)$, the two other equivalences (\ref{eqpqr2}) and (\ref{eqpqr3}) are obtained by considering the composition of $T$ with itself two and three times respectively. We note also that $T^{(4)}=Id$.
\end{proof}
The dimension $n$ being given and the three equivalences (\ref{eqpqr1}), (\ref{eqpqr2}) and (\ref{eqpqr3}) holding, each internal intersection point is generated through only two straight lines which means that its multiplicity is equal to 2. Indeed, for $i\in\{0,\ldots,n-1\}$, the only two straight lines which generate $z_{i,i+p,i+p+r,i+p+q+r}$ are $\mathcal{D}_{i,i+p}$ and $\mathcal{D}_{i+p+r,i+p+q+r}$ no matter which equivalence is chosen, provided the shift $\rho$ mentioned in the equivalence is used wisely. 	
\begin{corollary}
	If $p+q=n$ then $\mathcal{Q}$ is complex and
	\begin{alignat}{2}
		(p,q,r)=(p,n-p,r) & \equiv(n-p,n-p,-n+p-r) \quad & & [p+r] \label{eqppr1a}\\
		 & \equiv (n-p,p,r) \quad & & [p]\label{eqppr2a}\\
		 & \equiv (p,p,-p-r) \quad & &  [r]\label{eqppr3a}		
	 \end{alignat}
\end{corollary}
The proof is obvious by observing first that when $p+q=n$ then $r=-s$, i.e. $\mathcal{Q}$ is complex and next, by replacing $q$ by $n-p$ in Proposition \ref{prop4.4}. This result is not only a corollary of Proposition \ref{prop4.4}, it shows also that under specific conditions, two triplets $(p,q,r)$ and $(q,p,r)$ may be equivalent.
\begin{proposition}
If $n$ is even, we get the particular following equivalences:
\begin{alignat}{2}
	\left(p,\frac{n}{2},r\right) & \equiv \left(\frac{n}{2},p,r\right) \quad & & \left[(p+r)-\frac{n}{2}\right]\label{nsur21}\\	
	& \equiv \left(p,\frac{n}{2},r+\frac{n}{2}\right) \quad & & [0]\label{nsur22}\\
	& \equiv \left(p,p,2r\right) \quad & & \left[0\right]\label{nsur23}
\end{alignat}
\label{prop4.5}	
\end{proposition}
The proof of these results is obvious by using definitions (\ref{dP1P2}) and (\ref{AP1P2}) and trigonometric identities. 
\begin{remark}
	The dimension $n$ even being given and the two equivalences (\ref{nsur21}) and (\ref{nsur23}) holding, each intersection point is generated through at most three straight lines which means that its multiplicity is less or equal to 3. Indeed, for $i\in\{0,\ldots,n-1\}$, the only three straight lines which may generate $z_{i,i+p,i+p+r,i+p+\frac{n}{2}+r}$ are $\mathcal{D}_{i,i+p}$, $\mathcal{D}_{i+p+r,i+p+\frac{n}{2}+r}$ and $\mathcal{D}_{i+p+2r,i+2p+2r}$ no matter which equivalence is chosen, provided the mentioned shift is used. Nevertheless, it may happen that $i+p+2r=i\Leftrightarrow p+2r\equiv 0\mbox{ mod }n$ and in this case, only two straight lines generate the intersection point. Finally $z_{i,i+p,i+p+r,i+p+\frac{n}{2}+r}$ is of multiplicity $3$ iff $p+2r\neq 0$ modulo $n$.	
\end{remark}
So far, the relationships occuring in Propositions \ref{prop4.2}, \ref{prop4.3} and especially \ref{prop4.4} state that the multiplicity of any intersection point generated through these equivalences does not exceed 3. Nevertheless, Poonen and Rubinstein have proved in \cite{PR}, for example when $n=12$, that the multiplicity of $12$ intersection points is equal to $4$, which suggests that there exist other equivalences between the triplets $(p,q,r)$. However they look very particular and specific to some particular values of $n$. For example, by studying carefully the cases when $n$ is a multiple of $12$, we may state the following results.
\begin{proposition}
	If $n$ is a multiple of $12$, additionally to the equivalences generated through Propositions \ref{prop4.2}, \ref{prop4.3} and \ref{prop4.4}, we have
	\begin{itemize}
		\item if $\mathcal{Q}$ is simple
	\begin{alignat}{2}
		\left(\frac{n}{12},\frac{5n}{12},\frac{2n}{12}\right) & \equiv \left(\frac{7n}{12},\frac{2n}{12},\frac{n}{12}\right)\quad & & \left[\frac{-4n}{12}\right]\label{n12a}\\
		& \equiv \left(\frac{n}{12},\frac{2n}{12},\frac{3n}{12}\right)\quad & & \left[0\right]\label{n12b}\\
		\left(\frac{n}{12},\frac{2n}{12},\frac{2n}{12}\right) & \equiv \left(\frac{n}{12},\frac{7n}{12},\frac{n}{12}\right)\quad & & \left[0\right]\label{n12c}\\
		& \equiv \left(\frac{2n}{12},\frac{5n}{12},\frac{4n}{12}\right)\quad & & \left[\frac{3n}{12}\right]\label{n12d}\\
		\left(\frac{n}{12},\frac{n}{12},\frac{3n}{12}\right) & \equiv \left(\frac{n}{12},\frac{4n}{12},\frac{2n}{12}\right)\quad & & \left[0\right]\label{n12e}\\
		& \equiv \left(\frac{n}{12},\frac{8n}{12},\frac{n}{12}\right)\quad & & \left[0\right]\label{n12f}\\
		& \equiv \left(\frac{4n}{12},\frac{4n}{12},\frac{n}{12}\right)\quad & & \left[\frac{-2n}{12}\right]\label{n12g}				
	\end{alignat}
		\item if $\mathcal{Q}$ is complex
	\begin{alignat}{2}			
		\left(\frac{4n}{12},\frac{4n}{12},\frac{-n}{12}\right) & \equiv \left(\frac{4n}{12},\frac{3n}{12},\frac{-2n}{12}\right)\quad & & \left[0\right]\label{n12h}\\
		& \equiv \left(\frac{3n}{12},\frac{4n}{12},\frac{-2n}{12}\right)\quad & & \left[\frac{2n}{12}\right]\label{n12i}\\
		\left(\frac{7n}{12},\frac{7n}{12},\frac{-3n}{12}\right) & \equiv \left(\frac{8n}{12},\frac{8n}{12},\frac{-7n}{12}\right) \quad & & \left[\frac{n}{12}\right]\label{n12j}\\
		& \equiv \left(\frac{5n}{12},\frac{8n}{12},\frac{-2n}{12}\right)\quad & & \left[-\frac{n}{12}\right]\label{n12k}\\
		& \equiv \left(\frac{8n}{12},\frac{5n}{12},\frac{-2n}{12}\right)\quad & & \left[\frac{n}{12}\right]\label{n12l}\\		
		& \equiv \left(\frac{7n}{12},\frac{8n}{12},\frac{-5n}{12}\right)\quad & & \left[0\right]\label{n12m}\\
		& \equiv \left(\frac{8n}{12},\frac{7n}{12},\frac{-5n}{12}\right)\quad & & \left[\frac{n}{12}\right]\label{n12n}		
	\end{alignat}	
\end{itemize}
\label{prop4.6}
\end{proposition}
Note that all these equivalences are still true when we permute the two first components in each triplet provided the shift is adapted. For example, $\left(\frac{5n}{12},\frac{n}{12},\frac{2n}{12}\right)\equiv \left(\frac{2n}{12},\frac{7n}{12},\frac{n}{12}\right)\quad\left[\frac{2n}{12}\right]$.
\begin{proof}
	Let us show that (\ref{n12a}) holds, the other relationships being proved similarly. We set $(p_1,q_1,r_1)=\left(\frac{n}{12},\frac{5n}{12},\frac{2n}{12}\right)$, $(p_2,q_2,r_2)=\left(\frac{7n}{12},\frac{2n}{12},\frac{n}{12}\right)$ and then, we may easily show that $J_{n}(p_1,q_1,r_1)=J_{n}(p_2,q_2,r_2)=4+\sqrt{3}$. Next, we compute 
		\begin{itemize}
		\item[.] $\sin((p_1+q_1+2r_1)\frac{\pi}{n})=\sin(\frac{10\pi}{12})=\frac12$,
		\item[.] $\sin((p_2+q_2+2r_2)\frac{\pi}{n})=\sin(\frac{11\pi}{12})=\frac{\sqrt{2-\sqrt{3}}}{2}$,
		\item[.] $\cos((p_1-p_2)\frac{\pi}{n})=\cos(-\frac{6\pi}{12})=0$,
		\item[.] $-\cos((p_1-(2p_2+q_2+2r_2))\frac{\pi}{n})=-\cos(-\frac{17\pi}{12})=\cos(\frac{5\pi}{12})=\frac{\sqrt{2-\sqrt{3}}}{2}$,
		\item[.] $-\cos((2p_1+q_1+2r_1)-p_2)\frac{\pi}{n})=-\cos(\frac{4\pi}{12})=-\frac12$,
		\item[.] $\cos((2(p_1-p_2)+(q_1-q_2)+2(r_1-r_2))\frac{\pi}{n})=\cos(-\frac{7\pi}{12})=-\cos(\frac{5\pi}{12})=-\frac{\sqrt{2-\sqrt{3}}}{2}$.
	\end{itemize}
Then it holds
\[d_{P_1,P_2}=J^\frac12\cos^{-1}\left(\frac{1}{(4+\sqrt{3})\times\frac12\times\frac{\sqrt{2-\sqrt{3}}}{2}}\times\begin{pmatrix}-\frac{\sqrt{2-\sqrt{3}}}{2} & \frac{\sqrt{3}}{2}\end{pmatrix}
\begin{pmatrix}-\frac{\sqrt{2-\sqrt{3}}}{2} & \frac{\sqrt{2-\sqrt{3}}}{2}\\-\frac12 & 0\end{pmatrix}\begin{pmatrix}\frac{\sqrt{2+\sqrt{3}}}{2}\\\frac{\sqrt{2-\sqrt{3}}}{2}\end{pmatrix}\right)\]
\[=J^\frac12\cos^{-1}\left(\frac{1}{(4+\sqrt{3})\times\frac12\times\frac{\sqrt{2-\sqrt{3}}}{2}}\times\left(-\frac{\sqrt{2-\sqrt{3}}}{2}\left(-\frac14+\frac{2-\sqrt{3}}{4}\right)-\frac{\sqrt{3}}{2}\times\frac12\frac{\sqrt{2+\sqrt{3}}}{2}\right)\right)\]
\[=J^\frac12\cos^{-1}\left(\frac{1}{4(4+\sqrt{3})\sqrt{2-\sqrt{3}}}\left(-3\sqrt{6}+\sqrt{2}\right)\right)=J^\frac12\cos^{-1}\left(-\frac12\right)=J^\frac12\times4k\left(\frac{2\pi}{n}\right)\]
and equivalence (\ref{n12a}) is proved since $\rho=4k=\frac{4n}{12}$.
\end{proof}
These equivalences highlight several interesting points. First, help to the theory developed previously and as the following section will show, we have an easy access to the different quadrangle configurations giving rise to some particular situations such as an orbit containg only $n$ points or an orbit on which points have multiplicities greater than 2. Second, our method is able obviously to detect the orbits where the multiplicities are greater or equal to 3 and which are not due to Proposition \ref{prop4.5}. To our knowledge, the equations occuring in Proposition \ref{prop4.6} do not depend on more general relationships and explain why situations inducing multiplicities greater than 3 occur while being sporadic.\\
So if we look carefully the equations in Proposition \ref{prop4.6} when $\mathcal{Q}$ is simple, we note that (\ref{n12a}) and (\ref{n12b}) yield intersection points of multiplicity 3 implying $\mathcal{D}_{i,i+\frac{n}{12}}$, $\mathcal{D}_{i+\frac{3n}{12},i+\frac{8n}{12}}$ and $\mathcal{D}_{i+\frac{4n}{12},i+\frac{6n}{12}}$, $i\in\{0,\ldots,n-1\}$, (\ref{n12c}) and (\ref{n12d}) yield intersection points of multiplicity 3 implying $\mathcal{D}_{i,i+\frac{n}{12}}$, $\mathcal{D}_{i+\frac{3n}{12},i+\frac{5n}{12}}$ and $\mathcal{D}_{i+\frac{2n}{12},i+\frac{9n}{12}}$, $i\in\{0,\ldots,n-1\}$, while (\ref{n12e}), (\ref{n12f}) and (\ref{n12g}) yield intersection points of multiplicity 4 implying $\mathcal{D}_{i,i+\frac{n}{12}}$, $\mathcal{D}_{i+\frac{4n}{12},i+\frac{5n}{12}}$, $\mathcal{D}_{i+\frac{3n}{12},\frac{7n}{12}}$ and $\mathcal{D}_{i+\frac{2n}{12},i+\frac{10n}{12}}$, $i\in\{0,\ldots,n-1\}$. When $\mathcal{Q}$ is complex, we note that (\ref{n12h}) and (\ref{n12i}) yield intersection points of multiplicity 3 implying $\mathcal{D}_{i,i+\frac{4n}{12}}$, $\mathcal{D}_{i+\frac{3n}{12},i+\frac{7n}{12}}$ and $\mathcal{D}_{i+\frac{2n}{12},i+\frac{5n}{12}}$, $i\in\{0,\ldots,n-1\}$, while (\ref{n12j}), (\ref{n12k}), (\ref{n12l}), (\ref{n12m}) and (\ref{n12n}) yield intersection points of multiplicity 4 implying $\mathcal{D}_{i,i+\frac{7n}{12}}$, $\mathcal{D}_{i+\frac{4n}{12},i+\frac{11n}{12}}$, $\mathcal{D}_{i+\frac{n}{12},\frac{9n}{12}}$ and $\mathcal{D}_{i+\frac{2n}{12},i+\frac{10n}{12}}$, $i\in\{0,\ldots,n-1\}$.
These results indicate that external points may also have multiplicities greater than 2 and that for $n=12$ the maximal multiplicities of internal and external points are equal.
%

\section{\textit{Simuorb}, an algorithm based on a triple loop for generating all the orbits and all the intersection points}

This section is devoted to presenting the algorithm \textit{Simuorb} used to make the orbits and the intersection points generated by a clique-arrangement available as fast and as reliable as possible. Consecutively to the introduction of the three parameters $p,q,r$ in a previous section, it seems natural help to a triple loop to generate all the possible triplets $(p,q,r)$ associated to any cyclotomic clique-arrangement first to determine all the distinct orbits and next to provide their cardinalities as well as the multiplicities of the points lying on these orbits.

The intent of the following result is to provide a combinatorial description of the quadrangles $\mathcal{Q}$ with respect to triplets $(p,q,r)$, $s$ being known because of its definition (\ref{defpqrs}). In this way we substitute to the admissible quadruplets $(i,j,k,\ell)$ the triplets $(p,q,r)$ lying in the conbinatorial set defined as follows. 
\begin{theorem}
	Let us consider the admissible quadruplet $(i,j,k,\ell)$ and $\mathcal{Q}=(z_{i},z_{j},z_{k},z_{\ell})$ its associated quadrilateral.
	\begin{itemize}
		\item If $\mathcal{Q}$ is simple, we have $1\leq p\leq n-3$, $1\leq r\leq n-p-2$, $1\leq q\leq n-p-r-1$.
		\item If $\mathcal{Q}$ is complex (Case 2),
		\begin{itemize}
			\item[.] if $\sharp\left(\curvearc{z_iz_j}\cap\{z_k,z_\ell\}\right)=0$, we have $1\leq p\leq n-3$,
			\small{\[r\in\left\{\begin{array}{rcl}
					\llbracket 2,n-p-1\rrbracket & \mbox{ if } & p\geq\left[\frac{n-1}{2}\right]\vspace{0.1cm}\\
					\llbracket-\left[\frac{n-1}{2}\right],-p-1\rrbracket\cup\llbracket 2,\left[\frac{n}{2}\right]\rrbracket & \mbox{ if } & p<\left[\frac{n-1}{2}\right]
				\end{array}\right.,~q\in\left\{\begin{array}{rcl}\llbracket n-r+1,n-1\rrbracket & \mbox{ if } & r>0\\
					\llbracket -r+1,n-1\rrbracket & \mbox{ if } & r<0 \end{array}\right.,\]}
			\item[.] if $\sharp\left(\curvearc{z_iz_j}\cap\{z_k,z_\ell\}\right)=2$, we have $3\leq p\leq n-1$,
			\small{\[r\in\left\{\begin{array}{rcl}
					\llbracket -p+1,-2\rrbracket & \mbox{ if } & p\leq\left[\frac{n}{2}\right]+1\vspace{0.1cm}\\
					\llbracket-\left[\frac{n}{2}\right],-2\rrbracket\cup\llbracket n-p+1,\left[\frac{n-1}{2}\right]\rrbracket & \mbox{ if } & p>\left[\frac{n}{2}\right]+1
				\end{array}\right.,~q\in\left\{\begin{array}{rcl}\llbracket 1,-r-1\rrbracket & \mbox{ if } & r<0\\
					\llbracket 1,n-r-1\rrbracket & \mbox{ if } & r>0 \end{array}\right.,\]}			
		\end{itemize}
		\item If $\mathcal{Q}$ is complex (Case 3), we have $2\leq p\leq n-2$,
		\begin{itemize}
			\item[.] if $z_k\in\curvearc{z_iz_j}$, 
			\small{\[r\in\left\{\begin{array}{rcl}
					\llbracket -p+1,-1\rrbracket & \mbox{ if } & p\leq\left[\frac{n}{2}\right]+1\vspace{0.1cm}\\
					\llbracket-\left[\frac{n}{2}\right],-1\rrbracket\cup\llbracket n-p+1,\left[\frac{n-1}{2}\right]\rrbracket & \mbox{ if } & p>\left[\frac{n}{2}\right]+1
				\end{array}\right.,~q\in\left\{\begin{array}{rcl}\llbracket -r+1,n-r-p-1\rrbracket & \mbox{ if } & r<0\\
					\llbracket n-r+1,2n-r-p-1\rrbracket & \mbox{ if } & r>0 \end{array}\right.,\]}
			\item[.] if $z_k\in\curvearc{z_jz_i}$,
			\small{\[r\in\left\{\begin{array}{rcl}
					\llbracket 1,\left[\frac{n}{2}\right]\rrbracket\cup\llbracket -p-1,-\left[\frac{n-1}{2}\right]\rrbracket & \mbox{ if } & p<\left[\frac{n-1}{2}\right]\vspace{0.1cm}\\
					\llbracket 1,n-p-1\rrbracket & \mbox{ if } & p\geq\left[\frac{n-1}{2}\right]
				\end{array}\right.,~q\in\left\{\begin{array}{rcl}\llbracket -r-p+1,-r-1\rrbracket & \mbox{ if } & r<0\\
					\llbracket n-p-r+1,n-r-1\rrbracket & \mbox{ if } & r>0 \end{array}\right.,
				\]}		
		\end{itemize}
	\end{itemize}
	\label{thm5.1}
\end{theorem}
\begin{proof} 
	When the quadrilateral is simple, we have $p+q+r+s=\delta(i,j)+\delta(j,k)+\delta(k,\ell)+\delta(\ell,i)$. Because each summand is greater than or equal to 1, we get easily $1\leq p\leq n-3$, $1\leq r\leq n-p-2$ and $1\leq q\leq n-p-r-1$.\\
	Next, let us consider when $\mathcal{Q}=(z_i,z_j,z_k,z_\ell)$ is complex. In Case 2, we have to consider if $\sharp\left(\curvearc{z_iz_j}\cap\{z_k,z_\ell\}\right)=0$ or not. In the first case, $1\leq p\leq n-3$ and $r$ is positive if $p\geq\left[\frac{n-1}{2}\right]$ while $r$ may be positive or negative if $p<\left[\frac{n-1}{2}\right]$. Indeed, if $p\geq\left[\frac{n-1}{2}\right]$ then $n-p-1\leq\left[\frac{n}{2}\right]$ and  $2\leq\delta(j,k)\leq n-p-1$, which means that we get in any case $r=\delta(j,k)>0$ by using the two first results of Lemma \ref{lemma2.5}. The values $\left[\frac{n}{2}\right]$ and $\left[\frac{n-1}{2}\right]$ play an essential role in the proof since they allow to address the issue of the parity of $n$ and they avoid somes situations when $r$ may write differently, i.e. when $n$ is even and $r=\frac{n}{2}=-\frac{n}{2}\mbox{ mod }n$. If $p<\left[\frac{n-1}{2}\right]$, $2\leq \delta(j,k)$ may reach the value $\left[\frac{n}{2}\right]$ as well as more distant values by performing necessarily a clock-wise rotation. Because of the definition of  $r=\Delta(j,k)$, the allowed positive values for $-r$ are in this case such that $-r\in\{p+1,\ldots,\left[\frac{n-1}{2}\right]\}$. The possible values for $q$ are those which allow the vertex $z_\ell$ to lie between $z_j$ and $z_k$. The proof when $\sharp\left(\curvearc{z_iz_j}\cap\{z_k,z_\ell\}\right)=2$ is quite similar and just as subtle. In Case 3, a first necessary condition for the straight lines $\mathcal{D}_{i,j}$ and $\mathcal{D}_{k,\ell}$ to intersect inside the unit circle is that $2\leq p\leq n-2$. Indeed, because $z_k$ or $z_l$ lies in $\curvearc{z_iz_j}$ or $\curvearc{z_jz_i}$ (or because $z_i$ or $z_j$ lie in $\curvearc{z_kz_\ell}$ or $\curvearc{z_\ell z_k}$), neither $z_k$ nor $z_\ell$ (or neither $z_i$ nor $z_j$) are direct neighbours, enforcing $p$ (and $q$) to be greater than or equal to 2 and less or equal to $n-2$. By assuming next that $z_k\in\curvearc{z_iz_j}$, the value of $r$ depends necessarily on the value of $p$. Indeed, if $p\leq\left[\frac{n}{2}\right]+1$ then $\delta(j,k)>\left[\frac{n-1}{2}\right]$ and consequently, $r=-\delta(j,k)$ and $-p+1\leq r\leq -1$. On the other hand, if $p>\left[\frac{n}{2}\right]+1$, $r$ may be negative but also positive since $\delta(j,k)$ may be lower than or equal to  $\left[\frac{n-1}{2}\right]$. Then we have $r\in\{-\left[\frac{n}{2}\right],\ldots,-1\}\cup\{n-p+1,\ldots,\left[\frac{n-1}{2}\right]\}$. Once $p$ and $r$ have been chosen, depending on the sign of $r$, we have to reach the vertex $z_\ell$ counter clock-wise, by first passing through $z_j$ which explains the possible values for $q$. The case $z_k\in\curvearc{z_jz_i}$ may be treated in the same manner.		 
\end{proof}

 In the case of external orbits, we shall use the following routine
\begin{algorithm}[!ht]
	\caption{Triple loop for generating triplets $(p,q,r)$ associated to external orbits (Case 1)}
	\begin{algorithmic}[1]
		\Inputs{$n\geq 5$}
		\For{$p = 1$ to $n-3$}
		\For{$r = 1$ to $n-p-2$}
		\For{$q = 1$ to $n-p-q-1$}
		\State{$pq2r=p+q+2*r$}
		\If{$mod(pq2r,n)\sim=0$}
		\State $orb_{-}ext=[orb_{-}ext ; [p,q,r,\sqrt{Jn(p,q,r)}]]$
		\EndIf
		\EndFor
		\EndFor
		\EndFor	
		\State Remove identical quadruplets in $orb_{-}ext$
		\State Sort quadruplets in $orb_{-}ext$ w.r.t. decreasing values of $\sqrt{J_n}$
	\end{algorithmic}
	\label{algo1}
\end{algorithm}

We represent in Figures \ref{figJnext1} and \ref{figJnext2} respectively the values of $\sqrt{J_{20}(p,q,r)}$ with respect to the unfolding of the triple loop given in Algorithm \ref{algo1} and the values of $|r-s|$ (which indicate the level of parallelism between $\mathcal{D}_{i,j}$ and $\mathcal{D}_{k,\ell}$) with respect to the values of $\sqrt{J_{20}(p,q,r)}$.
\begin{figure}[!ht]
	\begin{minipage}[c]{.46\linewidth}
		\centering\includegraphics[width=5.5cm,height=4cm]{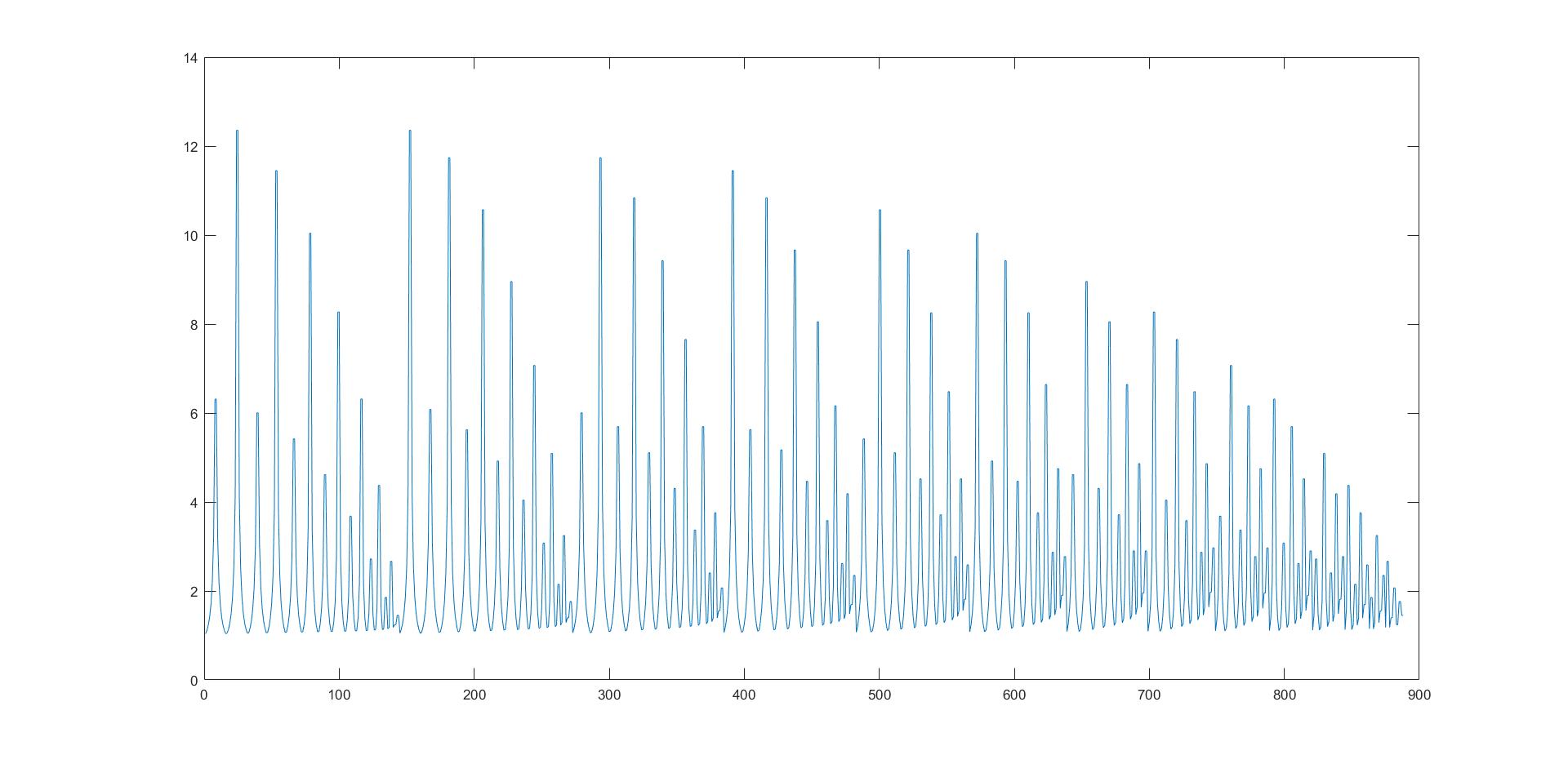}
		\caption{\centering Values of $\sqrt{J_{20}(p,q,r)}$ w.r.t. the unfolding of the triple loop (Algorithm \ref{algo1})}
		\label{figJnext1}
	\end{minipage} \hfill
	\begin{minipage}[c]{.46\linewidth}
		\centering\includegraphics[width=5.5cm,height=4cm]{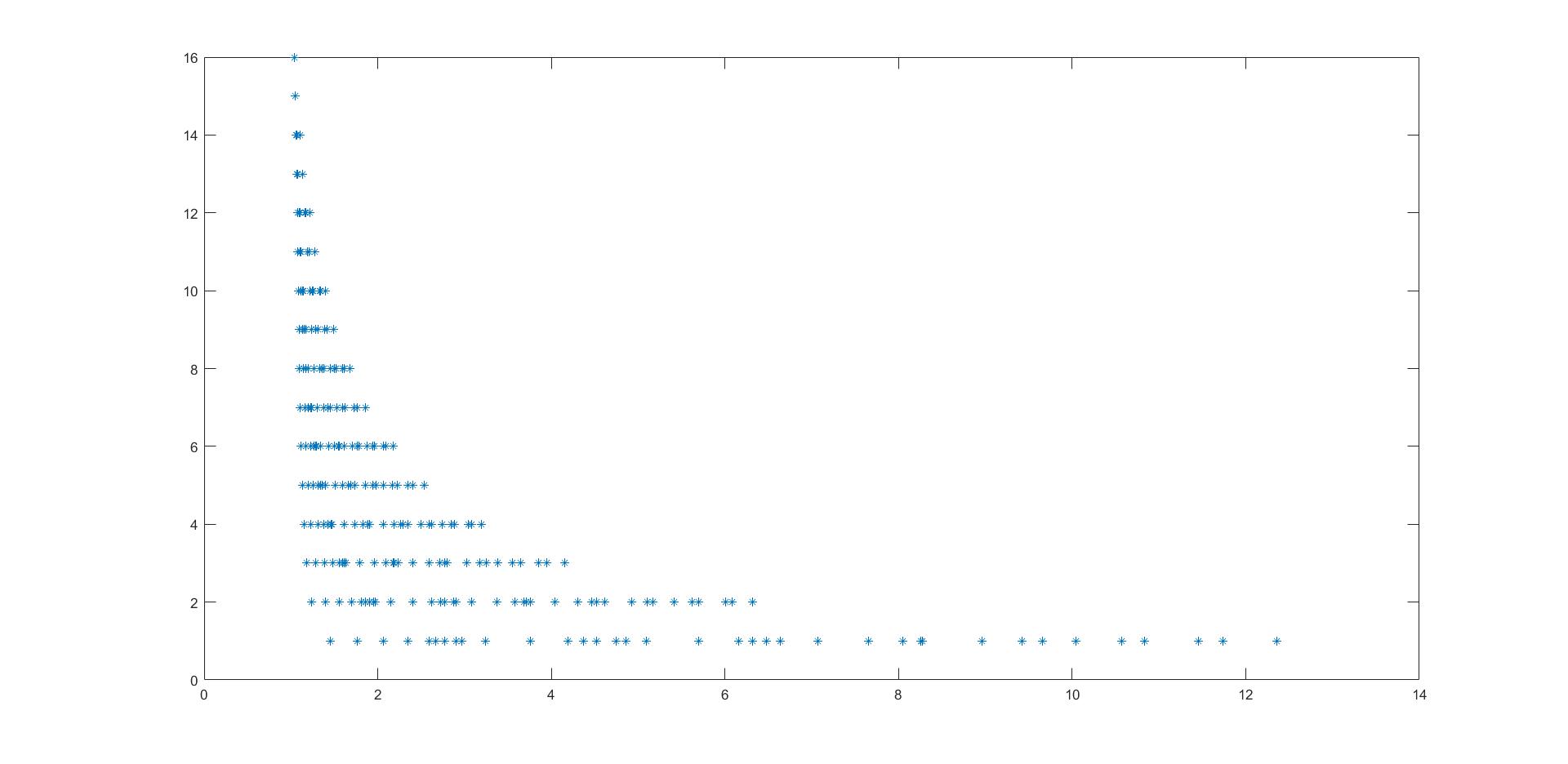}
		\caption{\centering Values of $|r-s|$ w.r.t. $\sqrt{J_{20}(p,q,r)}$}		
		\label{figJnext2}		
	\end{minipage}
\end{figure}

%
%

The greatest values of $\sqrt{J_n}$ are obtained when $\mathcal{D}_{i,j}$ and $\mathcal{D}_{k,\ell}$ are almost parallel and the most distant from each other, i.e. when $r$ ans $s$ are distinct but almost equal and as large as possible. For example, when $n=20$, $\max(\sqrt{J_n(p,q,r)})\simeq 12,3551$ is obtained through the triplets $(1,2,8)\equiv (2,1,9)\equiv (1,2,9)\equiv (2,1,8)$ where $|r-s|=1$. Next, the values of $\sqrt{J_n}$ close but greater to $1$ are obtained when $\mathcal{D}_{i,j}$ and $\mathcal{D}_{k,\ell}$ are almost identical which occurs in many situations. When $n=20$, $\min_{\sqrt{J_n}>1}(\sqrt{J_n(p,q,r)})\simeq 1,0385$ is obtained through the triplets $(1,1,1)\equiv (1,1,17)$ where $|r-s|=14$. Nevertheless, as Figure \ref{figJnext2} shows, the monotonicity of $|r-s|$ is not a sufficient condition to ensure the monotonicity of $\sqrt{J_n}$.\\

Help to Theorem \ref{thm5.1} we may use also Algorithms \ref{algo1a} and \ref{algo1b} which are mathematically equivalent to Algorithm \ref{algo1}.\\

\begin{algorithm}[!ht]
	\caption{Triple loop for generating triplets $(p,q,r)$ associated to external orbits  (Case 2 and $\sharp\left(\curvearc{z_iz_j}\cap\{z_k,z_\ell\}\right)=0$)}
		\begin{algorithmic}[1]
		\Inputs{$n\geq 5$}
		\State $n1=\left[\frac{n}{2}\right]$, $n2=\left[\frac{n-1}{2}\right]$
		\For{$p = 1$ to $n-3$}
			\If{$p\geq n2$}
				\State $rtmp=2:n-p-1$
			\Else
				\State $rtmp=[-n2:-p-1,2:n1]$
			\EndIf
			\For{$i=1:length(rtmp)$}
				\State $r=rtmp(i)$
				\If{$r<0$}
					\For{$q=-r+1:n-1$}
						\State $pq2r=p+q+2*r$
						\If{$mod(pq2r,n)\sim=0$}
							\State $orb_{-}ext=[orb_{-}ext ; [p,q,r,\sqrt{Jn(p,q,r)}]]$
						\EndIf
					\EndFor
				\Else
					\For{$q=n-r+1:n-1$}
						\State $pq2r=p+q+2*r$
						\If{$mod(pq2r,n)\sim=0$}
							\State $orb_{-}ext=[orb_{-}ext ; [p,q,r,\sqrt{Jn(p,q,r)}]]$
						\EndIf
					\EndFor
				\EndIf
			\EndFor
		\EndFor	
		\State Remove identical quadruplets in $orb_{-}ext$
		\State Sort quadruplets in $orb_{-}ext$ w.r.t. decreasing values of $\sqrt{J_n}$
		\end{algorithmic}
	\label{algo1a}
\end{algorithm}
\begin{algorithm}[!ht]
	\caption{Triple loop for generating triplets $(p,q,r)$ associated to external orbits  (Case 2 and $\sharp\left(\curvearc{z_iz_j}\cap\{z_k,z_\ell\}\right)=2$)}
	\begin{algorithmic}[1]
		\Inputs{$n\geq 5$}
		\State $n1=\left[\frac{n}{2}\right]$, $n2=\left[\frac{n-1}{2}\right]$
		\For{$p = 3$ to $n-1$}
			\If{$p\geq n1+1$}
				\State $rtmp=-p+1:-2$
			\Else
				\State $rtmp=[-n1:-2,n-p+1:n2]$
			\EndIf
			\For{$i=1:length(rtmp)$}
				\State $r=rtmp(i)$
				\If{$r<0$}
					\For{$q=1:-r-1$}
						\State $pq2r=p+q+2*r$
						\If{$mod(pq2r,n)\sim=0$}
							\State $orb_{-}ext=[orb_{-}ext ; [p,q,r,\sqrt{Jn(p,q,r)}]]$
						\EndIf
					\EndFor
				\Else
					\For{$q=1:n-r-1$}
						\State $pq2r=p+q+2*r$
						\If{$mod(pq2r,n)\sim=0$}
							\State $orb_{-}ext=[orb_{-}ext ; [p,q,r,\sqrt{Jn(p,q,r)}]]$
						\EndIf
					\EndFor
				\EndIf
			\EndFor
		\EndFor	
		\State Remove identical quadruplets in $orb_{-}ext$
		\State Sort quadruplets in $orb_{-}ext$ w.r.t. decreasing values of $\sqrt{J_n}$
	\end{algorithmic}
	\label{algo1b}
\end{algorithm}
		
In the case of internal orbits (Case 3), we shall use the Algorithms \ref{algo2a} or \ref{algo2b}.\\

\begin{algorithm}[!ht]
	\caption{Triple loop for generating triplets $(p,q,r)$ associated to internal orbits (Case 3 and $z_k\in\curvearc{z_iz_j}$)}
	\begin{algorithmic}[1]
		\Inputs{$n\geq 5$}
		\State $n1=\left[\frac{n}{2}\right]$, $n2=\left[\frac{n-1}{2}\right]$
		\For{$p = 2$ to $n-2$}
			\If{$p\geq n1+1$}
				\State $rtmp=-p+1:-1$
			\Else
				\State $rtmp=[-n1:-1,n-p+1:n2]$
			\EndIf
			\For{$i=1:length(rtmp)$}
				\State $r=rtmp(i)$
				\If{$r<0$}
					\For{$q=-r+1:n-r-p-1$}
						\If{$p\sim=n/2~||~q\sim=n/2$}
							\State $orb_{-}int=[orb_{-}int ; [p,q,r,\sqrt{Jn(p,q,r)}]]$
						\EndIf
					\EndFor
				\Else
					\For{$q=n-r+1:2*n-r-p-1$}
						\If{$p\sim=n/2~||~q\sim =n/2$}
							\State $orb_{-}int=[orb_{-}int ; [p,q,r,\sqrt{Jn(p,q,r)}]]$
						\EndIf
					\EndFor
				\EndIf
			\EndFor
		\EndFor	
		\State Remove identical quadruplets in $orb_{-}int$
		\State Sort quadruplets in $orb_{-}int$ w.r.t. decreasing values of $\sqrt{J_n}$
	\end{algorithmic}
	\label{algo2a}
\end{algorithm}
\begin{algorithm}[!ht]
		\caption{Triple loop for generating triplets $(p,q,r)$ associated to internal orbits (Case 3 and $z_k\in\curvearc{z_jz_i}$)}
\begin{algorithmic}[1]
	\Inputs{$n\geq 5$}
	\State $n1=\left[\frac{n}{2}\right]$, $n2=\left[\frac{n-1}{2}\right]$
	\For{$p = 2$ to $n-2$}
		\If{$p\geq n2$}
			\State $rtmp=1:n-p-1$
		\Else
			\State $rtmp=[1:n1,-p-1:n2]$
		\EndIf
		\For{$i=1:length(rtmp)$}
			\State $r=rtmp(i)$
			\If{$r<0$}
				\For{$q=-r-p+1:-r-1$}
					\If{$p\sim=n/2~||~q\sim=n/2$}
						\State $orb_{-}int=[orb_{-}int ; [p,q,r,\sqrt{Jn(p,q,r)}]]$
					\EndIf
				\EndFor
			\Else
				\For{$q=n-p-r+1:n-r-1$}
					\If{$p\sim=n/2~||~q\sim=n/2$}
						\State $orb_{-}int=[orb_{-}int ; [p,q,r,\sqrt{Jn(p,q,r)}]]$
					\EndIf
				\EndFor
			\EndIf
		\EndFor
	\EndFor	
	\State Remove identical quadruplets in $orb_{-}int$
	\State Sort quadruplets in $orb_{-}int$ w.r.t. decreasing values of $\sqrt{J_n}$
\end{algorithmic}
\label{algo2b}
\end{algorithm}

As in the external case, we represent in Figures \ref{figJnint1} and \ref{figJnint2} respectively the values of $\sqrt{J_{20}(p,q,r)}$ with respect to the unfolding of the triple loop given in Algorithms \ref{algo2a} and \ref{algo2b} and the values of $p+q$ (which indicate whether $\mathcal{D}_{i,j}$ and $\mathcal{D}_{k,\ell}$ look like diameters of the unit circle) with respect to the values of $\sqrt{J_{20}(p,q,r)}$.\\
\begin{figure}[!ht]
	\begin{minipage}[c]{.46\linewidth}
		\centering\includegraphics[width=5.5cm,height=4cm]{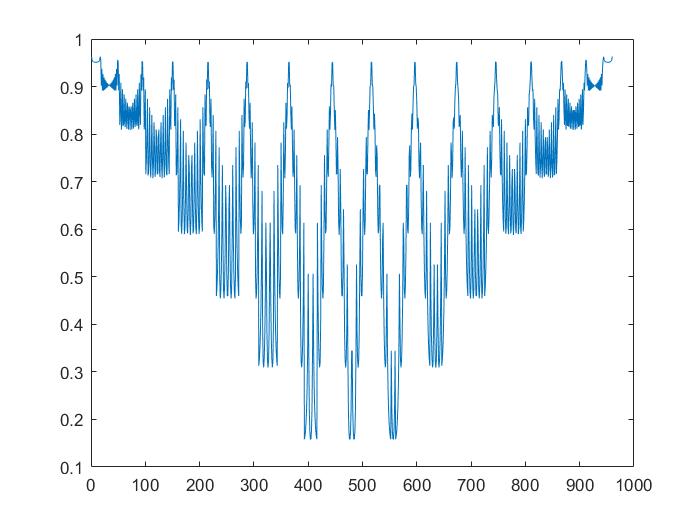}
		\caption{\centering Values of $\sqrt{J_{20}(p,q,r)}$ w.r.t. the unfolding of the triple loop (Algorithm \ref{algo2a} or \ref{algo2b})}
		\label{figJnint1}
	\end{minipage} \hfill
	\begin{minipage}[c]{.46\linewidth}
		\centering\includegraphics[width=5.5cm,height=4cm]{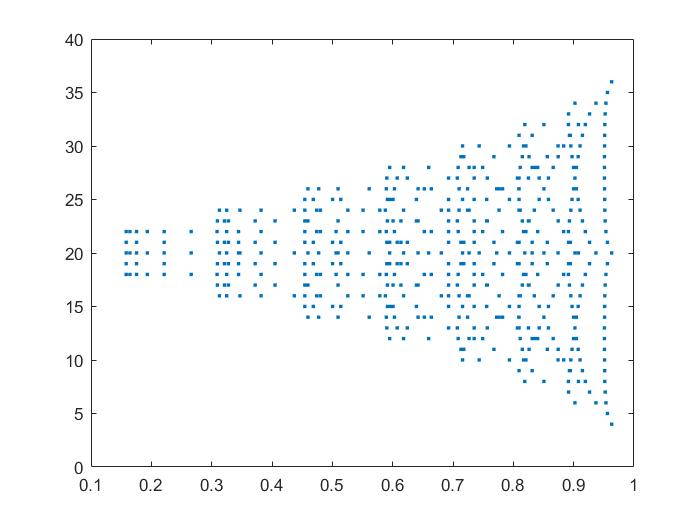}
		\caption{\centering Values of $p+q$ w.r.t. $\sqrt{J_{20}(p,q,r)}$}		
		\label{figJnint2}		
	\end{minipage}
\end{figure}
%
%
%

The values of $\sqrt{J_n}$ close but lower to $1$ are obtained when $\mathcal{D}_{i,j}$ and $\mathcal{D}_{k,\ell}$ are almost identical which occurs in many different situations. For example, when $n=20$, $\max_{\sqrt{J_n}<1}(\sqrt{J_n(p,q,r)})\simeq 0,9269$ is obtained through the triplets $(2,2,-1)\equiv (18,18,-17)\equiv (2,18,-1)\equiv (18,2,-1)$. Next, the smallest values of $\sqrt{J_n}$ are obtained when $\mathcal{D}_{i,j}$ and $\mathcal{D}_{k,\ell}$ are (almost) diameters of the unit circle, i.e. when $p$ and $q$ are close to $\frac{n}{2}$. For example, when $n=20$, $\min(\sqrt{J_n})\simeq 0,1584$ is reached through the triplets $(9,9,-8)\equiv(9,10,-4)\equiv (10,9,-4)\equiv (9,10,-5)\equiv(10,9,-5)\equiv (9,11,-1)\equiv (11,9,-1)\equiv (10,11,-6,)\equiv (11,10,-6)\equiv (11,11,-10)\equiv(10,11,-5)\equiv (11,10,-5)$.\\

We may easily prove the different results obtained in our previous examples when $n=20$. Let us assume temporarily in the following proposition that $p,q,r$ are continuous variables lying in the interval $]-n,n[$.
\begin{proposition}
	The mapping $J_n$ is not coercive and its critical points $(p,q,r)$ of $J_n$ are of the shape $(\frac{n}{2}+k_1n,\frac{n}{2}+k_2n,r)$ with $k_1,k_2\in\mathbb{Z}$ and $r$ an arbitrary real number such that $\frac{2r}{n}\notin\mathbb{Z}$.
	\label{prop5.1}
\end{proposition}
\begin{proof}
	By using Definition \ref{Jnmod}, we note that $J_n$ is positive.  
	Since we have 
	\begin{center}
		$J_n(p,q,r)=J_n(p+2kn,q,r)=J_n(p,q+2kn,r)=J_n(p,q,r+kn)$ for all integer $k\in\mathbb{Z}$,
	\end{center} 
we see that $J_n$ is periodic of group of period $(2n\mathbb{Z})^2 \times n\mathbb{Z}$ and thus, this entails coercivity of $J_n$. We derive next the necessary first order conditions for $J_n$. We have
	\begin{center}
		$\small \begin{array}{rcl}
			\displaystyle \frac{\partial J_n}{\partial p}(p,q,r) & = & \frac{-2\pi}{n\sin^3\left((p+q+2r)\frac{\pi}{n}\right)}\left[\sin\left(p\frac{\pi}{n}\right)\cos\left(p\frac{\pi}{n}\right)\sin\left((p+q+2r)\frac{\pi}{n}\right)-\right.\\
			&&\cos\left(q\frac{\pi}{n}\right)\sin\left(p\frac{\pi}{n}\right)\cos\left((p+q+2r)\frac{\pi}{n}\right)\sin\left((p+q+2r)\frac{\pi}{n}\right)-\cos\left(q\frac{\pi}{n}\right)\cos\left(p\frac{\pi}{n}\right)\sin^2\left((p+q+2r)\frac{\pi}{n}\right)+\\
			&&\left.\cos^2\left(p\frac{\pi}{n}\right)\cos\left((p+q+2r)\frac{\pi}{n}\right)+\cos^2\left(q\frac{\pi}{n}\right)\cos\left((p+q+2r)\frac{\pi}{n}\right)-2\cos\left(p\frac{\pi}{n}\right)\cos\left(q\frac{\pi}{n}\right)\cos^2\left((p+q+2r)\frac{\pi}{n}\right)\right]\\
			& = & \frac{-2\pi}{n\sin^3\left((p+q+2r)\frac{\pi}{n}\right)}\left[\cos\left(p\frac{\pi}{n}\right)\cos\left((q+2r)\frac{\pi}{n}\right)-\cos\left(p\frac{\pi}{n}\right)\cos\left(q\frac{\pi}{n}\right)-\cos\left(p\frac{\pi}{n}\right)\cos\left(q\frac{\pi}{n}\right)\cos^2\left((p+q+2r)\frac{\pi}{n}\right)-\right.\\
			&&\left.\cos\left(q\frac{\pi}{n}\right)\sin\left(p\frac{\pi}{n}\right)\cos\left((p+q+2r)\frac{\pi}{n}\right)\sin\left((p+q+2r)\frac{\pi}{n}\right)+\cos^2\left(q\frac{\pi}{n}\right)\cos\left((p+q+2r)\frac{\pi}{n}\right)\right]\\
			& = &\frac{-2\pi}{n\sin^3\left((p+q+2r)\frac{\pi}{n}\right)}\left[\cos\left(p\frac{\pi}{n}\right)\cos\left((q+2r)\frac{\pi}{n}\right)-\cos\left(p\frac{\pi}{n}\right)\cos\left(q\frac{\pi}{n}\right)-\right.\\
			&&\left.\cos\left(q\frac{\pi}{n}\right)\cos\left((p+q+2r)\frac{\pi}{n}\right)\cos\left((q+2r)\frac{\pi}{n}\right)+\cos^2\left(q\frac{\pi}{n}\right)\cos\left((p+q+2r)\frac{\pi}{n}\right)\right]\\
			& = & \frac{2\pi}{n\sin^3\left((p+q+2r)\frac{\pi}{n}\right)}\left(\cos\left(p\frac{\pi}{n}\right)-\cos\left(q\frac{\pi}{n}\right)\cos\left((p+q+2r)\frac{\pi}{n}\right)\right)\left(\cos\left(q\frac{\pi}{n}\right)-\cos\left((q+2r)\frac{\pi}{n}\right)\right).
		\end{array}$
	\end{center}
	Because $p$ and $q$ play a symmetric role in $J_n$, we have
	\begin{center}
		$\small \frac{\partial J_n}{\partial q}(p,q,r)=\frac{2\pi}{n\sin^3\left((p+q+2r)\frac{\pi}{n}\right)}\left(\cos\left(q\frac{\pi}{n}\right)-\cos\left(p\frac{\pi}{n}\right)\cos\left((p+q+2r)\frac{\pi}{n}\right)\right)\left(\cos\left(p\frac{\pi}{n}\right)-\cos\left((p+2r)\frac{\pi}{n}\right)\right).$
	\end{center}
	Next,
	\begin{center}
		$\small \begin{array}{rcl}
			\displaystyle \frac{\partial J_n}{\partial r}(p,q,r) & = & \frac{4\pi}{n\sin^3\left((p+q+2r)\frac{\pi}{n}\right)}\left[\cos\left(p\frac{\pi}{n}\right)\cos\left(q\frac{\pi}{n}\right)\sin^2\left((p+q+2r)\frac{\pi}{n}\right)-\cos^2\left(p\frac{\pi}{n}\right)\cos\left((p+q+2r)\frac{\pi}{n}\right)-\right.\\
			&&\left.\cos^2\left(q\frac{\pi}{n}\right)\cos\left((p+q+2r)\frac{\pi}{n}\right)+2\cos\left(p\frac{\pi}{n}\right)\cos\left(q\frac{\pi}{n}\right)\cos^2\left((p+q+2r)\frac{\pi}{n}\right)\right]\\
			& = & \frac{4\pi}{n\sin^3\left((p+q+2r)\frac{\pi}{n}\right)}\left[\cos\left(p\frac{\pi}{n}\right)\cos\left(q\frac{\pi}{n}\right)+\cos\left(p\frac{\pi}{n}\right)\cos\left(q\frac{\pi}{n}\right)\cos^2\left((p+q+2r)\frac{\pi}{n}\right)-\right.\\
			&&\left.\cos^2\left(p\frac{\pi}{n}\right)\cos\left((p+q+2r)\frac{\pi}{n}\right)-\cos^2\left(q\frac{\pi}{n}\right)\cos\left((p+q+2r)\frac{\pi}{n}\right)\right]\\
			& = & \frac{4\pi}{n\sin^3\left((p+q+2r)\frac{\pi}{n}\right)}\left(\cos\left(q\frac{\pi}{n}\right)-\cos\left(p\frac{\pi}{n}\right)\cos\left((p+q+2r)\frac{\pi}{n}\right)\right)\left(\cos\left(p\frac{\pi}{n}\right)-\cos\left(q\frac{\pi}{n}\right)\cos\left((p+q+2r)\frac{\pi}{n}\right)\right).
		\end{array}$
	\end{center}
	The three partial derivatives must cancel simultaneously so we have
	\[\left\{\begin{array}{rcl}
		\cos\left(p\frac{\pi}{n}\right) & = & \cos\left(q\frac{\pi}{n}\right)\cos\left((p+q+2r)\frac{\pi}{n}\right)\mbox{ and/or }\cos\left(q\frac{\pi}{n}\right)=\cos\left((q+2r)\frac{\pi}{n}\right)\\
		\cos\left(q\frac{\pi}{n}\right) & = & \cos\left(p\frac{\pi}{n}\right)\cos\left((p+q+2r)\frac{\pi}{n}\right)\mbox{ and/or }\cos\left(p\frac{\pi}{n}\right)=\cos\left((p+2r)\frac{\pi}{n}\right)\\
		\cos\left(q\frac{\pi}{n}\right) & = & \cos\left(p\frac{\pi}{n}\right)\cos\left((p+q+2r)\frac{\pi}{n}\right)\mbox{ and/or }\cos\left(p\frac{\pi}{n}\right)= \cos\left(q\frac{\pi}{n}\right)\cos\left((p+q+2r)\frac{\pi}{n}\right)
	\end{array}\right..\]
	Nevertheless, the conditions $\cos\left(q\frac{\pi}{n}\right)=\cos\left((q+2r)\frac{\pi}{n}\right)$ and $\cos\left(p\frac{\pi}{n}\right)=\cos\left((p+2r)\frac{\pi}{n}\right)$ are impossible according to the conditions over $p,q$ and $r$ so, in order to state the first-order optimality condition, we enforce 
	\[\cos\left(q\frac{\pi}{n}\right)=\cos\left(p\frac{\pi}{n}\right)\cos\left((p+q+2r)\frac{\pi}{n}\right)\mbox{ and }\cos\left(p\frac{\pi}{n}\right)= \cos\left(q\frac{\pi}{n}\right)\cos\left((p+q+2r)\frac{\pi}{n}\right).\]
	Since $J_n(p,q,r)$ is well-defined, we do not have $\cos((p+q+2r)\frac{\pi}{n})=\pm 1$ which implies that the critical points $(p,q,r)$ have the shape indicated.
\end{proof}
Coming back to our context, because $(p,q,r)$ is a lattice vector, $J_n$ is minimal if and only if $p=q=\frac{n}{2}$ if $n$ is even, and $p=\frac{n\pm 1}{2}$ and $q=\frac{n\mp 1}{2}$ if $n$ is odd. Note also that when the straight lines $(z_i,z_j)$ and $(z_k,z_\ell)$ are concurrent at the center of $\mathcal{C}_0$ we obtain a global strict minimum equal to 0 for $J_n$.\\

Once the triplets have been generated, we gather and we sort them with respect to $\sqrt{J_n}$ and we get interested to each sub-family of triplets associated to a particular value of $\sqrt{J_n}$. Help to the circular arc length $d_{P_1,P_2}$ we determine if the triplets are equivalent or not (according to the first point of Proposition \ref{prop4.1}) and we gather and we label them according to this condition. Then we deduce from these sub-collections the number of points on the orbit (according to the second point in Proposition \ref{prop4.1}) and their multiplicities. We develop this procedure in the following pseudocode.
\begin{algorithm}[!ht]
	\caption{Determination of the number of points of the clique-arrangement and their multiplicities}
	\begin{algorithmic}[1]
		\Inputs{$n\geq 5$, $orb_{-}ext$, $orb_{-}int$}
		\State Compute $nborb_{-}ext=\#(orb_{-}ext(4,:))$ and $nborb_{-}int=\#(orb_{-}int(4,:))$
		\For{$i = 1$ to $nborb_{-}ext$}
		\State Gather the quadruplets associated to $\sqrt{J_n}(i)$ in the matrix $miniorb_{-}ext(i)$
		\State Filter the redundant triplets according to equations (\ref{eqpqr0}), (\ref{eqppr1a}), (\ref{eqppr2a}) and (\ref{eqppr3a}).
		\State Compute the circular arc length (and the associated shift $\rho$) between all the vectors in $miniorb_{-}ext(i)$, label them if they are equivalent, and sort them w.r.t $label$ in the submatrices $subminiorb_{-}ext(i,j)$
		\State For each vector of $subminiorb_{-}ext(i,j)$ containing $p,q,r,\sqrt{J_n},\rho,label$, associate the vector $(\rho,\rho+p,\rho+p+r,\rho+p+q+r,\sqrt{J_n})$ and form the matrix $subminiorbind_{-}ext(i,j)$
		\State Compute the numbers $mult_{-}ext(i,j)$ of distinct pairs occuring in the first two couples of components in $subminiorbind_{-}ext(i,j)$
		\State Compute the maximal multiplicity of each orbit $mult_{-}ext(i)=max_j(mult_{-}ext(i,j))$
		\State Compute the number of nonzero (nnz) matrix elements in $mult_{-}ext(i)$ and the number of points on each orbit $nbpts_{-}ext(i)=nnz(mult_{-}ext(i)\times n)$		
		\EndFor	
		\For{$i = 1$ to $nborb_{-}int$}
		\State Gather the quadruplets associated to $\sqrt{J_n}(i)$ in the matrix $miniorb_{-}int(i)$
		\State Filter the redundant triplets according to equations (\ref{eqpqr0}) to (\ref{eqppr3a}).		
		\State Compute the circular arc length (and the associated shift $\rho$) between all the vectors in $miniorb_{-}int(i)$, label them if they are equivalent, and sort them w.r.t $label$ in the submatrices $subminiorb_{-}int(i,j)$
		\State For each vector of $subminiorb_{-}int(i,j)$ containing $p,q,r,\sqrt{J_n},\rho,label$, associate the vector $(\rho,\rho+p,\rho+p+r,\rho+p+q+r,\sqrt{J_n})$ and form the matrix $subminiorbind_{-}int(i,j)$
		\State Compute the numbers $mult_{-}int(i,j)$ of distinct pairs occuring in the first two couples of components in $subminiorbind_{-}int(i,j)$
		\State Compute the maximal multiplicity of each orbit $mult_{-}int(i)=max_j(mult_{-}int(i,j))$
		\State Compute the number of nonzero ($nnz$) matrix elements in $mult_{-}int(i)$ and the number of points on each orbit $nbpts_{-}int(i)=nnz(mult_{-}int(i)\times n)$		
		\EndFor			
		\State Compute the total number of external points $nbpts_{-}ext=\sum_i nbpts_{-}ext(i)$, the total number of internal points $nbpts_{-}int=\sum_i nbpts_{-}int(i)$ and the total number of intersection points $nbpts$.
	\end{algorithmic}		
	\label{algo3}
\end{algorithm}   

\section{Numerical considerations}

This last section is devoted to some numerical experiments to prove the efficiency of the procedure \textit{Simuorb} which consists of the use of Algorithms \ref{algo1} (or \ref{algo1a} or \ref{algo1b}), \ref{algo2a} (or \ref{algo2b}) and \ref{algo3}. All computations are carried out using MATLAB version 25.1.0.2973910 (R2025a), in double precision arithmetic, on a Dell Latitude 5310 laptop, with an Intel Core i5-10210U 1.6GHz processor, running under the Professional Windows 10 operating system.\\

As mentioned previously, we may complete the table 7 of \cite{PR} with the results related to the external orbits. As in \cite{PR}, $a_k(n)$ denotes the number of points inside the regular $n$-gon other than the center where exactly $k$ lines meet. Similarly we define $\tilde{a}_k(n)$ as the number of points outside the regular $n$-gon where exactly $k$ lines meet. We obtain the following datas
\begin{table}[!ht]
	\centering
	\scalebox{0.7}{	
	\begin{tabular}{|c||*{10}{c|}}
		\hline
		$n$ & $a_2/\tilde{a}_{2}$ & $a_3/\tilde{a}_3$ & $a_4/\tilde{a}_4$ & $a_5/\tilde{a}_5$ & $a_6/\tilde{a}_6$ & $a_7/\tilde{a}_7$ & $nbpts_{-}int/nbpts_{-}{ext}$ & $nbpts$ & $nborb_{-}{int}/nborb_{-}{ext}$ & $nborb$ \\
		\hline 
		4 &                       &                   &                   &                  &                  &                   &     1/0                   &   5     &       1/0                & 2 \\
		\hline
		5 &          5/5          &                   &                   &                  &                  &                   &     5/5                   &   15    &       1/1                & 3\\
		\hline
		6 &         12/18         &                   &                   &                  &                  &                   &     13/18                 &   37    &       3/2                & 6\\
		\hline
		7 &         35/49         &                   &                   &                  &                  &                   &     35/49                 &   91    &       4/5                & 10\\
		\hline
		8 &         40/80         &        8/8        &                   &                  &                  &                   &     49/88                 &   145   &       6/7                & 14\\ 
		\hline
		9 &         126/198       &                   &                   &                  &                  &                   &     126/198               &   333   &      10/14               & 25\\
		\hline
		10 &        140/280       &        20/20      &                   &                  &                  &                   &     161/300               &   471   &      13/18               & 32\\   
		\hline
		11 &        330/550       &                   &                   &                  &                  &                   &     330/550               &   891   &      20/30               & 51\\
		\hline
		12 &        228/480       &      60/96       &       12/12        &                  &                  &                   &     301/588               &   901   &      19/29               & 49\\
		\hline
		13 &       715/1235       &                  &                    &                  &                  &                   &     715/1235              &   1963  &      35/55               & 91\\
		\hline
		14 &       644/1246       &     112/168      &                    &                   &                  &                   &    757/1414             &   2185  &      37/58               & 96\\
		\hline
		15 &       1365/2415      &                  &                    &                  &                  &                    &    1365/2415            &   3795  &      53/87               & 141\\
		\hline
		16 &       1168/2240      &    208/336       &                    &                  &                  &                    &    1377/2576            &   3969  &      56/91               & 148\\
		\hline
		17 &       2380/4284      &                  &                    &                  &                  &                    &    2380/4284             &   6681 &      84/140              & 225\\
		\hline
		18 &       1512/3060      &    216/540       &       54/54        &     54/54        &                  &                    &    1837/3708             &   5563 &      65/114              & 179\\
		\hline
		19 &       3876/7068      &                  &                    &                  &                  &                    &    3876/7068             & 10963  &     120/204              & 325\\
		\hline
		20 &       3360/6480      &   480/800        &                    &                  &                  &                    &    3841/7280             & 11141  &      115/194             & 310\\
		\hline
		21 &     5985/11025      &                   &                    &                 &                   &                    &    5985/11025            & 17031  &     165/285              & 451\\
		\hline
		22 &       5280/10230    &   660/1100        &                    &                  &                  &                    &    5941/11330            & 17293  &    160/276               & 437\\
		\hline
		23 &     8855/16445      &                  &                     &                  &                  &                    &    8855/16445            & 25323  &     219/385              & 605\\
		\hline
		24 &     6144/12576     &   864/1656        &       264/336       &     24/24        &                  &                    &    7297/14592            &  21913 &    175/320               & 496\\
		\hline
		25 &    12650/23650     &                   &                     &                  &                  &                    &    12650/23650           & 36325  &    286/505               & 792\\
		\hline
		26 &    11284/21944      &   1196/2028       &                     &                  &                  &                    &    12481/23972           & 36479  &    276/489              &  766\\
		\hline
		27 &     17550/32994    &                   &                     &                   &                  &                    &   17550/32994            & 50571  &    364/650              & 1015\\
		\hline
		28 &    15680/30520     &    1568/2688      &                     &                   &                  &                    &   17249/33208            & 50485  &    346/625             & 972\\
		\hline
		29 &   23751/44863      &                   &                     &                   &                  &                    &   23751/44863           & 68643   &    455/819             & 1275\\
		\hline
		30 &    13800/29310     &    2250/4350     &     420/780         &      180/240      &    120/120       &    30/30           &   16801/34830            & 51661   &   319/611             & 931\\
		\hline   
	\end{tabular}}
\caption{A listing of $a_k(n)$, $\tilde{a}_k(n)$, $nbpts_{-}int(n)$, $nbpts_{-}{ext}(n)$, $nbpts(n)$, $nborb_{-}{int}(n)$, $nborb_{-}{ext}(n)$ and $nborb(n)$}
\label{table3}
\end{table}

We compute next the cpu-time ``$Time~[s]$'' needed to exceed $10^k$ intersection points, $k\in\{4,\ldots,8\}$ and to reach our maximum number of intersection points (approximatively equal to 280 millions which corresponds to a value $k_{max}\simeq 8,45$), the time spent to compute the external orbits "$Time_{-}ext$" and the internal orbits "$Time_{-}int$" as well as the value of $n$ needed to reach this amount of points.
\begin{table}[!ht]
	\centering
	\begin{tabular}{|c||*{6}{c|}}
		\hline
		$k$ & 4 & 5 & 6 & 7 & 8 & $k_{max}$\\
		\hline
		"$Time_{-}ext~[s]$" & 0.0926 & 0.4241 & 1.2276 & 6.8648 & 58.8097  & 222.6499\\
		\hline
		"$Time_{-}int~[s]$" & 0.0809 & 0.2462 & 0.7826 & 4.4920 & 27.5226 & 67.6834\\		
		\hline
		"$Time~[s]$"        & 0.1735 & 0.6703 & 2.0102 & 11.3568 & 86.3323 & 290.3334\\
		\hline
		$n$ & 19 & 33 & 55 & 97 & 171 & 220\\
		\hline
	\end{tabular}
\label{table4}
\caption{Time performances of \textit{Simuorb}}
\end{table}
Of course, the maximal number of intersection points mentioned in the table is reached only because of the limited capacities of our computer. We note that the time spent to treat the external orbits is shorter than the one to treat the internal orbits. We remark also that the time spent to generate the orbits through the different triple loops represents barely $1\%$ of the total time.


\end{document}